%% file: Paper.tex
\documentclass[a4paper,12pt]{article}
\usepackage{array,amsthm,amsfonts,epsfig}
\usepackage{graphicx}
\usepackage{dina42e}
\usepackage{amsmath}
\usepackage{mathdots}
\usepackage{amssymb}
\usepackage{textcomp}
\usepackage[T1]{fontenc}
\usepackage{fancyhdr}
\usepackage[utf8]{inputenc} %deutsche Umlaute für Linux
\usepackage{euscript} %Euler-Schriften
\usepackage{mathrsfs}  %Sschreibschrift f�r Hilbertr�ume
\usepackage{calligra} %Calligra-Schriften

\newtheorem{satz}{Theorem}[section]
\newtheorem{thm}[satz]{Theorem}
\newtheorem{lemma}[satz]{Lemma}
\newtheorem{kor}[satz]{Corollary}
\newtheorem{prop}[satz]{Proposition}

\theoremstyle{definition}
\newtheorem{Def}[satz]{Definition}
\theoremstyle{remark}
\newtheorem{bem}[satz]{Remark}
\newtheorem{bsp}[satz]{Example}

\newcommand{\dq}{\textrm{\textit{\dj}} }                                           % durchgestrichenes d
\newcommand{\varf}[2]{#1_1, \ldots, #1_{#2}}                              %z.B. s_1,...,s_l
\newcommand{\vara}[2]{#1_1 + \ldots + #1_{#2}}                            %z.B. s_1 + ... + s_l
\newcommand{\<}[1]{\langle #1 \rangle}                                    %<.> 
\newcommand{\op}{OP}                                             %OP
\newcommand{\e}{\varepsilon}                                              %\varepsilon
\DeclareMathOperator{\ad}{ad}                                             %funktion ad
\newcommand{\supp}{\textrm{supp }}                                        %Träger einer Funktion  %%%%% falls auch in kursiven Umgebungen eine gerade Schrift gewünscht wird, muss statt \textrm folgender Befehl verwendet werden: \textnormal
\newcommand{\adc}[2]{ \ad(-ix)^{#1_1}\ad(D_x)^{#2_1}\ldots \ad(-ix)^{#1_l}\ad(D_x)^{#2_l} }  %classifikations funktionen ad(-ix)ad(D_x)..ad(D_x)
\newcommand{\Sallg}[5]{S^{#1}_{#2,#3}(\R^{#4}\times\R^{#5})}              %glatte Symbole
\newcommand{\Sallgn}[6]{S^{#1}_{#2,#3}(\R^{#4}\times\R^{#5}; #6)}         %nichtglatte Symbole in \xi
\newcommand{\Sn}[3]{\Sallg{#1}{#2}{#3}{n}{n}}                             %glatte Symbole in R^n x R^n
\newcommand{\Snn}[4]{S^{#1}_{#2,#3}(\R^{n}\times\R^{n}; #4)}              %glatte Symbole in x, nicht glatt in \xi auf R^n x R^n
\newcommand{\pa}[1]{\partial_{\xi}^{#1}}                                  %Partielles Ableiten nach \xi
\newcommand{\p}{\partial}
\newcommand{\pax}[2]{\partial_{\xi}^{#1}\partial_{x}^{#2}}                %Partielles Ableiten nach \xi^{\alpha} und nach x \betha mal
\newcommand{\skh}[3]{{\langle #1,#2 \rangle}_{#3}}                        %Skalarprodukt < , > mit Hilbertraum-Index
\newcommand{\con}{C^{\infty}_c(\mathbb{R}^n)}                             %Raum C_0 unendlich von den reellen Zahlen^n
\newcommand{\R}{\mathbb{R}}                                               %R
\newcommand{\Rn}{\mathbb{R}^n}                                            %R^n
\newcommand{\RnRn}{\R^n \times \R^n}                                      %R^n x R^n
\newcommand{\RnRnx}[2]{\R^n_{#1} \times \R^n_{#2}}                         %R^n_y x R^n_z
\newcommand{\RnRnRn}{\RnRn \times \R^n}                                   %R^n x R^n x R^n
\newcommand{\RnRnRnRn}{\RnRn \times \RnRn}                                %R^n x R^n x R^n x R^n
\newcommand{\intr}{\int \limits_{\Rn} }                                   %Integral über die reellen Zahlen^n
\newcommand{\osint}{\textrm{Os\hspace*{0,1cm}-}\hspace{-0,15cm}\iint}     %Os-Integral 
\newcommand{\osiint}{\textrm{Os\hspace*{0,1cm}-}\hspace{-0,15cm}\iiiint}  %Os-Integral 
\newcommand{\s}{\mathcal{S}(\R^n)}                                        %Raum der schnell fallenden Funktionen
\newcommand{\sindo}[1]{\mathcal{S}(\R^n_{#1}) }                           %Raum der schnell fallenden Funktionen (R^n_y)
\newcommand{\snn}{\mathcal{S}(\RnRn)}                                     %Raum der schnell fallenden Funktionen (R^n,R^n)
\newcommand{\sd}{\mathcal{S'}(\R^n)}                                      %Raum der schnell fallenden Funktionen
\newcommand{\N}{\mathbb{N}}                                               %Natürliche Zahlen
\newcommand{\Z}{\mathbb{Z}}                                               %ganze Zahlen
\newcommand{\Non}{\mathbb{N}_0^n}                                               %Natürliche Zahlen_0^n
\newcommand{\C}{\mathbb{C}}                                               %Komplexe Zahlen
\title{Characterization of Non-Smooth Pseudodifferential Operators}
\author{Helmut Abels and Christine Pfeuffer}

\begin{document}

\maketitle

\begin{abstract}
  Smooth pseudodifferential operators on $\Rn$ can be characterized by their mapping properties between $L^p-$Sobolev spaces due to Beals and Ueberberg.
  In applications such a characterization would also be useful in the non-smooth case, for example to show the regularity of solutions of a partial differential equation.  
  Therefore, we will show that every linear operator $P$, which satisfies some specific continuity assumptions, is a non-smooth pseudodifferential operator of the symbol-class $C^{\tau} S^m_{1,0}(\mathbb{R}^n \times \mathbb{R}^n)$. 
  The main new difficulties are the limited mapping properties of pseudodifferential operators with non-smooth symbols.
\end{abstract}

\input{introduction}

\input{Preliminaries}

\input{Characterization}
\input{SymbolComposition}

%\bibliographystyle{plain}
%\bibliography{mybibtex} %\nocite{*}

\end{document}

%% file: introduction.tex
\section{Introduction}

In the smooth case some characterizations of pseudodifferential operators are already proved e.g.\;for the Hörmander class $\Sn{m}{\rho}{\delta}$. Here a smooth function $p$ is an element of the symbol-class $\Sn{m}{\rho}{\delta}$ with $m \in \R$ and $0 \leq \delta \leq \rho \leq 1$ if and only if
\begin{align*}
  |p|^{(m)}_k := \max_{|\alpha|,|\beta|\leq k} \sup_{x, \xi \in \R^n}|\pax{\alpha}{\beta} p(x,\xi)|\<{\xi}^{-(m-\rho|\alpha|+\delta|\beta|)} < \infty
\end{align*}
holds for all $k \in \N_0$. The associated pseudodifferential operator is defined by
\begin{align*}
		\op (p)(x):= p(x,D_x) u(x) := \intr e^{ix \cdot \xi} p(x,\xi) \hat{u}(\xi) \dq \xi \qquad \text{for all } u \in \s , x \in \Rn.
\end{align*}
Here $\s$ denotes the Schwartz space, the space of all rapidly decreasing smooth functions and $\hat{u}$ is the Fourier transformation of $u$. We also denote $\hat{u}$ via $\mathscr{F}[u]$. The set of all pseudodifferential operators with symbols in the symbol-class $\Sn{m}{\rho}{\delta}$ is denoted via $\op \Sn{m}{\rho}{\delta}$. \\

In 1977 Beals \cite{Beals} proved a characterization of smooth pseudodifferential operators, for example of the Hörmander class $\Sn{m}{\rho}{\delta}$ with $0 \leq \delta \leq \rho \leq 1$ and $\delta<1$. 
Later Ueberberg \cite{Ueberberg} generalized this characterization for $L^p-$Sobolev spaces for pseudodifferential operators of the Hörmander class $\Sn{m}{\rho}{\delta}$ with $0 \leq \delta \leq \rho \leq 1$ and $\delta<1$.
In the literature there are some other characterizations in the smooth case, e.g.\;\cite{Kryakvin}, \cite{SchroheLeopold1993} or \cite{Schrohe1990}. But the most important one for this section is that one of Ueberberg, cf.\;\cite{Ueberberg}. It is based on the method for characterizing algebras of pseudodifferential operators developed by Beals \cite{Beals}, \cite{Beals3},  Coifman, Meyer \cite{CoifmanMeyer} and Cordes \cite{Cordes1}, \cite{Cordes2}. Since non-smooth 
pseudodifferential operators are used in the regularity theory for partial differential equations, such a characterization 
is also useful in 
the non-smooth case. We use the main ideas of the characterization of Ueberberg in the smooth case, cf.\;\cite{Ueberberg}, in order to derive a characterization for non-smooth pseudodifferential operators of the symbol-class $p \in C^{\tau}_{\ast} \Snn{m}{\rho}{0}{M}$ with $\rho \in \{ 0,1 \}$. Here the Hölder-Zygmund space $C^{\tau}_{\ast}(\Rn)$, $\tau >0$, is defined by
\begin{align*}
  C^{\tau}_{\ast}(\Rn):=\left\{ f \in \sd: \|f\|_{C^{\tau}_{\ast} }:= \sup_{j \in \N_0} 2^{js} \|\mathscr{F}^{-1}[\varphi_j \hat{f}] \|_{L^{\infty}} < \infty \right\},
\end{align*}
where $\mathscr{F}^{-1}[u]$ is the inverse Fourier Transformation of $u \in \sd$, the dual space of $\s$. Moreover a function $p: \RnRn \rightarrow \C$ is an element of the symbol-class $ C^{\tau}_{\ast} \Snn{m}{\rho}{0}{M}$ with $m \in \R$, $\tau>0$, $0 \leq \rho \leq 1$ and $M\in \N_0 \cup \{ \infty \}$ if and only if
  \begin{itemize}

		\item[i)] $\p_x^{\beta} p(x, .) \in C^M(\Rn)$ for all $x \in \Rn$,

		\item[ii)] $\p_x^{\beta} \pa{\alpha} p \in C^{0}(\R^n_x \times \R^n_{\xi})$,

		\item[iv)] $\| \pa{\alpha} p(.,\xi)  \|_{C^{\tau}_{\ast}(\R^n)} \leq C_{\alpha}\<{\xi}^{\tilde{m}-\rho|\alpha|}$ for all $\xi \in \R^n$

	\end{itemize}
holds for all $\alpha, \beta \in \N_0^n$ with $|\alpha| \leq M$ and $|\beta| \leq m$. The associated pseudodifferential operator $p(x,D_x)$, also denoted by $\op(p)$, to such a symbol $p$ is defined in the same way as in the smooth case. The set of all pseudodifferential operators with symbols in the symbol-class $C^{\tau}_{\ast} \Snn{m}{\rho}{0}{M}$ is denoted via $\op C^{\tau}_{\ast} \Snn{m}{\rho}{0}{M}$.\\

Motivated by the characterization of Ueberberg \cite{Ueberberg} and Beals \cite{Beals} in the smooth case, we define the following set of operators:

\begin{Def}
	Let $m\in\R$, $M \in \N_0 \cup \{ \infty \}$ and $ 0\leq \rho \leq 1 $. Additionally let $\tilde{m} \in \N_0 \cup \{ \infty \}$ and $1 < q < \infty$. Then we define 
	$\mathcal{A}^{m,M}_{\rho,0}(\tilde{m},q)$
	as the set of all linear and bounded operators $P: H^m_q(\Rn) \rightarrow L^q(\Rn)$, such that for all $l\in\N$, $\alpha_1, \ldots, \alpha_l \in \Non$ and $ \beta_1, \ldots, \beta_l \in \N_0^n$ with $|\alpha_1| + |\beta_1|= \ldots = |\alpha_l| + |\beta_l|= 1$, $|\alpha| \leq M$ and $|\beta|\leq \tilde{m}$ the iterated commutator of $P$
	\begin{eqnarray*}
		\adc{\alpha}{\beta} P: H^{m-\rho|\alpha|}_q(\Rn) \rightarrow L^q(\Rn)
	\end{eqnarray*}
	is continuous. Here $\alpha:= \alpha_1 + \ldots + \alpha_l$ and $\beta:=\beta_1 + \ldots + \beta_l$. 
\end{Def}

For the definition of the iterated commutators, see Definition \ref{Def:IteratedCommutators} below. In the case $M = \infty$ we write $\mathcal{A}^{m}_{\rho,0}(\tilde{m},q)$ instead of $\mathcal{A}^{m,\infty}_{\rho,0}(\tilde{m},q)$.\\

Choosing $M=\tilde{m}=\infty$ the proof of the characterization in the smooth case, cf.\;\cite[Chapter 3]{Ueberberg} provides that each $T \in \mathcal{A}^{m,\infty}_{\rho,0}(\infty,q)$ is a smooth pseudodifferential operator of the class $\Sn{m}{\rho}{0}$. But we even get more: Smooth pseudodifferential operators of the symbol-class $ \Sn{m}{\rho}{0}$ are elements of $\mathcal{A}^{m}_{\rho,0}(\infty,q)$ due to Remark \ref{bem:SymbolOfIteratedCommutator} and Theorem \ref{thm:stetigInBesselPotRaum}. Thus we have $\mathcal{A}^{m}_{\rho,0}(\infty,q) = \op\Sn{m}{\rho}{0}$. In the case $\tilde{m} \neq \infty$ we obtain a similar result: Non-smooth pseudodifferential operators of the class $C^{\tau}_{\ast} S^m_{\rho, 0}(\RnRn)$ with $\rho \in \{ 0,1 \}$ are elements of such sets:

\begin{bsp}\label{bsp:ElementDerCharakterisierungsmenge}
  Let $\tau > 0$, $\tau \notin \N$, $m \in \R$ and $\rho \in \{ 0, 1 \}$. Considering a non-smooth symbol $p \in C^{\tau}_{\ast} \Sn{m}{\rho}{0}$ we get for $\tilde{m}:= \max\{ k \in \N_0: \tau-k>n/2 \}$ and $1 < q < \infty$:
  \begin{enumerate}
    \item[i)] $p(x, D_x) \in \mathcal{A}^{m+n/2}_{0,0}(\lfloor \tau \rfloor,2)$ if $\rho = 0$,
    \item[ii)] $p(x, D_x) \in \mathcal{A}^{m}_{0,0}(\tilde{m},2)$ if $\rho = 0$,
    \item[iii)] $p(x, D_x) \in \mathcal{A}^m_{1,0}(\lfloor \tau \rfloor ,q)$ if $\rho = 1$.
  \end{enumerate}
\end{bsp}

As in the smooth case the characterization of non-smooth pseudodifferential operators of the symbol-class $C^{\tau}_{\ast} \Snn{m}{1}{0}{M}$ is reduced to the characterization of those ones of the symbol-class $C^{\tau}_{\ast} \Snn{m}{0}{0}{M}$. To this end the following property of the set $\mathcal{A}^{m,M}_{\rho,0}(\tilde{m},q)$ is needed:

\begin{lemma} \label{lemma:setA}
	Let $m\in\R$, $M \in \N_0 \cup \{ \infty \}$ and $ 0\leq \rho_1 < \rho_2 \leq 1 $. Furthermore, let $\tilde{m} \in \N_0$ and $1 < q < \infty$. Then
	$$\mathcal{A}^{m, M}_{\rho_2,0}(\tilde{m},q) \subseteq \mathcal{A}^{m, M}_{\rho_1,0}(\tilde{m},q).$$
\end{lemma}

  \begin{proof}
    On account of the continuous embedding $H^{m-\rho_1|\alpha|}_q(\Rn) \hookrightarrow H^{m-\rho_2|\alpha|}_q(\Rn)$
    the claim holds.
  \end{proof}

The main goal of this paper is to show that each element of $\mathcal{A}^{m,M}_{1,0}(\tilde{m},q)$ is a non-smooth pseudodifferential operator with coefficients in a Hölder space. This is the topic of Subsection \ref{section:classificationA10}. We will see that $M$ has to be sufficiently large. In analogy to the proof of J.\,Ueberberg in the smooth case one reduces this statement to the following: Each element of the set $\mathcal{A}^{m,M}_{0,0}(\tilde{m},q)$ is a non-smooth pseudodifferential operator with coefficients in a Hölder space. Making use of order reducing pseudodifferential operators we obtain the characterization of non-smooth pseudodifferential operators of arbitrary order $m$ from that. Details for deriving this result are explained in Subsection \ref{section:classificationA00}. 

The first three subsections of Section \ref{Section:Charakterisierung} serve to develop some auxiliary tools needed for the proof of the case $m=0$. In Subsection \ref{section:pointwiseConvergence} we start by showing that a bounded sequence in $C^{\tilde{m}, s} \Snn{0}{0}{0}{M}$ has a subsequence which converges in the symbol-class $C^{\tilde{m}, s} \Snn{0}{0}{0}{M-1}$. Subsection \ref{section:symbolReduction} is devoted to the symbol reduction of non-smooth double symbols to non-smooth single symbols. Details for the third tool are proved in Subsection \ref{section:PropertiesOfTe}. There a family $( T_{\e} )_{\e \in (0,1]}$ fulfilling the following three properties is constructed: $T_{\e}: \sd \rightarrow \s$ is continuous for all $\e \in (0,1]$ and converges pointwise if $\e \rightarrow 0$. Moreover, all iterated commutators of $T_{\e}$ are uniformly bounded with respect to $\e$ as maps from $L^q(\Rn)$ to $L^q(\Rn)$. 

In Section \ref{section:AnwendungCaracterization} we illustrate the usefulness of such a characterization: We show that the composition $PQ$ of two non-smooth 
pseudodifferential operators $P$ and $Q$ is a non-smooth pseudodifferential operator again if $Q$ is smooth enough. This is done by means of the characterization of non-smooth pseudodifferential operators. 

Section \ref{Kapitel: PDO} is devoted to some properties of pseudodifferential operators with single symbols, cf.\,Subsection \ref{Section:PropertiesOfPDO}, and with double symbols, cf.\,Subsection \ref{Section:DoubleSymbols}. 

All notations and basics needed in this paper are introduced in Section \ref{section:Preliminaries}.\\
This paper is based on a part of the PhD-thesis of the second author, cf.\,\cite{Diss} advised by the first author. \\
\textbf{Acknowledgement:} We would like to thank Prof. Dr. Elmar Schrohe for his helpful suggestions to improve this paper.

%% file: Preliminaries.tex
\section{Preliminaries}\label{section:Preliminaries}

During the whole paper, we consider $n \in \N$ except when stated otherwise. In particular $n \neq 0$. Considering $x \in \R$ we define 
\begin{align*}
  x^+:=\max \{0;x \} \qquad \text{and} \qquad \lfloor x \rfloor:= \max \{k \in \Z : k \leq x \}.
\end{align*}
Moreover $$\<{x}:=(1+|x|^2)^{1/2} \quad \text{for all } x \in \Rn \qquad \text { and } \qquad  \dq \xi:= (2 \pi)^{-n} d \xi.$$
Additionally we scale partial derivatives with respect to a variable $x\in \Rn$ with the factor $-i$ and denote it by 
$$D_x^{\alpha}:= (-i)^{|\alpha|} \p^{\alpha}_x := (-i)^{|\alpha|} \p^{\alpha_1}_{x_1} \ldots \p^{\alpha_n}_{x_n} .$$ 
Here $\alpha =(\alpha_1, \ldots, \alpha_n) \in \Non$ is a \textit{multi-index}. 
For arbitrary $j\in \{1,\ldots, n\}$ we define the $j$-th canonical unit vector $e_j \in \N^n_0$ as $(e_j)_k =1$ if $j=k$ and $(e_j)_k=0$ else. 

In view of two Banach spaces $X,Y$ the set $\mathscr{L}(X,Y)$
consists of all linear and bounded operators $A:X \rightarrow Y$. We also write $\mathscr{L}(X)$ 
instead of $\mathscr{L}(X,X)$. \\
% Furthermore, $\mathrm{GL}(n)$ is the set of all invertible $n\times n$-matrices.\\

We finally note that the dual space of a topological vector space $V$ is denoted by $V'$.
If $V$ is even a Banach space the duality product $V$ is denoted by $\skh{.}{.}{V; V'}$.

\subsection{Functions on $\Rn$ and Function Spaces}

In this subsection we are going to introduce some function spaces which play a central role during this paper. To begin with, the \textit{Hölder space}  of the order $m \in \N_0$ with Hölder continuity $s \in (0,1]$ is denoted by $C^{m,s}(\Rn)$ and also by $C^{m+s}(\Rn)$ if $s \neq 1$. Note, that $C^{s}(\Rn) = C^s_{\ast}(\Rn)$ if $s \notin \N_0$. For arbitrary $s \in \R$ and $1<p<\infty$ the \textit{Bessel Potential Space} $H^s_p(\Rn)$ is defined by
\begin{align*}
  H^s_p(\Rn):= \{ f \in \sd: \<{D_x}^s f \in L^p(\Rn) < \infty \}
\end{align*}
where $\<{D_x}^s:=\op(\<{\xi}^s) $.\\

Let us mention a characterization of functions in a Bessel potential space needed later on:

\begin{lemma}\label{lemma:CharacterizationOfBesselPotentialSpaces}
  Let $1<p<\infty$, $s < 0$ and $m := -\lfloor s \rfloor$. Then for each $f \in H^s_p(\Rn)$ there are functions $g_{\alpha} \in H^{s-\lfloor s \rfloor}_p(\Rn)$, where $\alpha \in \Non$ with $|\alpha| \leq m$, such that
  \begin{itemize}
    \item $f= \sum \limits_{|\alpha| \leq m} \p^{\alpha}_x g_{\alpha}$, 
    \item $\sum \limits_{|\alpha| \leq m} \| g_{\alpha} \|_{H^{s-\lfloor s \rfloor}_p} \leq C \| f \|_{H^s_p}$,
  \end{itemize}
  where $C$ is independent of $f$ and $g_{\alpha}$.
\end{lemma}

\begin{proof}%[Proof of Lemma \ref{lemma:CharacterizationOfBesselPotentialSpaces}:]
  We define the operator $T: H^{-s}_q(\Rn) \rightarrow \left( H^{\lfloor s \rfloor -s}_q(\Rn) \right)^N$ with $1/p + 1/q =1 $ and $N:= \sharp \{ \alpha \in \Non : |\alpha| \leq m \}$ in the following way: 
  \begin{align*}
    T(\varphi) = ( \p_x^{\alpha}  \varphi )_{|\alpha| \leq m}.
  \end{align*}
  The norm $\| . \|_{X_q}$ of $X_q:=\left( H^{\lfloor s \rfloor -s}_q(\Rn) \right)^N$ is defined by 
  \begin{align*}
    \| f \|_{ X_q } &: = \sum_{i=1}^N \|f_i\|_{H^{\lfloor s \rfloor -s}_q} \qquad  \text{for all }f \in \left( H^{\lfloor s \rfloor -s}_q (\Rn) \right)^N .
  \end{align*}
  By means of an application of the bounded inverse theorem
  we can show the existence of the inverse of $T: H^{-s}_q(\Rn) \rightarrow T(H^{-s}_q(\Rn))=:Y$ and that $T^{-1} \in \mathscr{L}(Y, H^{-s}_q(\Rn))$. For more details see \cite[Proposition 2.49]{Diss}.
  In view of $T^{-1} \in \mathscr{L}(Y,H^{-s}_q(\Rn))$ we get
  \begin{align*}
    |\tilde{f}(g)| 
    = |f \circ T^{-1} (g)| 
    = |\skh{f}{T^{-1} g}{ H^s_p; H^{-s}_q }| 
    \leq \| f \|_{H^s_p} \| T^{-1} g \|_{H^{-s}_q} 
    \leq C \| f \|_{H^s_p} \| g \|_{ X_q}
  \end{align*}
  for all $f \in  H^s_p(\Rn)$ and $g \in Y$.
  An application of the Theorem of Hahn Banach provides the existence of a linear functional $F \in \left( X_q \right)' = \left( H_p^{s-\lfloor s \rfloor} (\Rn)\right)^N$ such that
  \begin{itemize}
    \item[$i)$] $F|_{Y} = \tilde{f}$,
    \item[$ii)$] $|F(g)| \leq C \| f \|_{H^s_p} \| g \|_{ X_q}$ for all $g \in X_q$.
  \end{itemize}
  For arbitrary $\varphi \in \s \subseteq H^{-s}_q(\Rn)$ we can apply property $i)$ on account of $T \varphi \in Y$ and get:
  \begin{align*}
    \langle  \sum_{|\alpha| \leq m} (-1)^{|\alpha|} \p^{\alpha}_x F^{\alpha} , \varphi  \rangle_{H^s_p; H^{-s}_q}
    &= \langle  F, \left( \p^{\alpha}_x \varphi \right)_{|\alpha| \leq m}  \rangle_{ \left( X_q \right)'; X_q} 
    = \langle  F, T \varphi  \rangle_{ \left( X_q \right)'; X_q} \\
    &= \langle  \tilde{f} , T \varphi  \rangle_{ \left( X_q \right)'; X_q}
    = f \circ T^{-1} (T \varphi) 
    = \skh{f}{\varphi}{H^s_p; H^{-s}_q}.
  \end{align*}
  Due to the density of $\s$ in $ H^{-s}_q(\Rn)$
%   , cf.\;Lemma \ref{lemma:EigenschaftenVomBesselPotentialRaum} 
  we obtain 
  \begin{align*}
    \sum_{|\alpha| \leq m} (-1)^{|\alpha|} \p^{\alpha}_x F^{\alpha} = f \qquad \text{in } H^s_p(\Rn).
  \end{align*}
  Additionally we get the claim due to $ii)$:
  \begin{align*}
    \sum_{|\alpha| \leq m} \| (-1)^{|\alpha|}  F^{\alpha} \|_{H^{s-\lfloor s \rfloor}_p}
    = \| F \|_{\left( X_q \right)'}
    = \sup_{ \| g\|_{X_q} \leq 1 } | \skh{F}{g}{ \left( X_q \right)'; X_q} |
    \leq C \| f \|_{H^s_p}.
  \end{align*}
  \vspace*{-1cm}

\end{proof}

In this paper the translation function $\tau_y(g): \Rn \rightarrow \C$, $y \in \Rn$  of $g \in L^1(\Rn)$, is defined for all $ x\in \Rn$ as
    $\tau_y(g)(x):=g(x-y).$

\subsection{Kernel Theorem}

An important ingredient of the characterization is the next kernel theorem:

\begin{thm}\label{thm:SchwartzKernel}
	Every continuous linear operator $T: \sd \rightarrow \s$ has a Schwartz kernel $t(x,y) \in \snn$. Thus for every $u \in \s$ we have
	\begin{align*}
		Tu(x) = \intr t(x,y) u(y) dy \qquad \text{for all } x \in \Rn.
	\end{align*}
\end{thm}

\begin{proof}
  This claim is a consequence of \cite[Theorem 51.6]{Treves} and \cite[Theorem 1.48]{Amann2} if one uses that $\s$ and $\sd$ are nuclear and conuclear, see e.g. \cite[p.\,530]{Treves}. For more details we refer to \cite[Theorem 2.62]{Diss}.
\end{proof}

We can apply the previous kernel theorem for the iterated commutators of linear and bounded operators $P: \sd \rightarrow \s$. These operators are defined in the following way:

\begin{Def}\label{Def:IteratedCommutators}
  Let $X,Y \in \{ \s, \sd \}$ and $T: X \rightarrow Y$ be linear. We define the linear operators $\ad(-ix_j) T: X \rightarrow Y$ and $\ad(D_{x_j}) T: X \rightarrow Y$ for all $j\in \{ 1, \ldots, n\}$ and $u \in X$ by
  \begin{eqnarray*}
    \ad(-ix_j) T u:= -ix_j Tu + T \left( ix_j u \right) \quad \textrm{and} \quad \ad(D_{x_j}) T u:= D_{x_j} \left( T u \right) - T \left( D_{x_j} u \right).
  \end{eqnarray*}
  For arbitrary multi-indices $\alpha, \beta \in \N_0^n$ we denote the \textit{iterated commutator} of $T$ as
  \begin{eqnarray*}
    \ad(-ix)^{\alpha}\ad(D_{x})^{\beta} T := [\ad(-ix_1)]^{\alpha_1} \ldots [\ad(-ix_n)]^{\alpha_n} [\ad(D_{x_1})]^{\beta_1} \ldots [\ad(D_{x_n})]^{\beta_n} T.
  \end{eqnarray*}
\end{Def}

On account of the previous definition all iterated commutators of $P: \sd \rightarrow \s$ map $\sd$ to $\s$. Consequently an application of Theorem \ref{thm:SchwartzKernel} provides:

\begin{kor}\label{kor:KernFürAdP}
  Let $\alpha, \beta \in \Non$ and $P: \sd \rightarrow \s$ be a linear operator. Then the operator $\ad(-ix)^{\alpha} \ad(D_x)^{\beta} P: \sd \rightarrow \s$ has a Schwartz kernel $f^{\alpha,\beta} \in \snn$, i.e., for all $u \in \s$ 
    \begin{align}
      \ad(-ix)^{\alpha} \ad(D_x)^{\beta} P u(x) = \intr f^{\alpha, \beta}(x,y) u(y) dy \qquad \text{for all } x \in \Rn.
    \end{align}
\end{kor}

Iterated commutators of pseudodifferential operators are pseudodifferential operators again due to the properties of the Fourier transformation:

\begin{bem}\label{bem:SymbolOfIteratedCommutatorNonSmooth}\label{bem:SymbolOfIteratedCommutator}
  Let $\tilde{m}\in \N_0$, $M \in \N_0 \cup \{ \infty \}$, $0< \tau \leq 1$, $m\in \R$ and $0 \leq \rho \leq 1$. We assume that $ p \in C^{\tilde{m}, \tau} S^m_{\rho,0} (\RnRn; M)$. Moreover, let $l \in \N$, $\alpha_1, \ldots, \alpha_l \in \Non$ and $\beta_1, \ldots, \beta_l \in \Non$ with $|\alpha_j + \beta_j| = 1$ for all $j \in \{ 1, \ldots, n\}$, $|\alpha| \leq M$ and $|\beta| \leq \tilde{m}$. Here $\alpha$ and $\beta$ are defined by $\alpha:= \alpha_1 + \ldots + \alpha_l$ and $\beta := \beta_1 + \ldots + \beta_l$. Then the operator 
  $$\ad(-ix)^{\alpha_1} \ad(D_x)^{\beta_1} \ldots \ad(-ix)^{\alpha_l} \ad(D_x)^{\beta_l} p(x,D_x)$$ 
  is a pseudodifferential operator with the symbol 
  $$\pa{\alpha} D^{\beta}_x  p(x,\xi) \in C^{\tilde{m}- |\beta|, \tau} S^{m-\rho |\alpha|}_{\rho,0} (\RnRnx{x}{\xi}; M-|\alpha|).$$
  If we even have $ p \in S^m_{\rho,0} (\RnRn)$, then $\pa{\alpha} D^{\beta}_x  p(x,\xi) \in S^{m-\rho |\alpha| }_{\rho,0} (\RnRn)$.
\end{bem}

An application of the kernel theorem provides:

\begin{lemma}\label{lemma:AbleitungVonp}
  Let $g \in \s$. For all $y \in \Rn$ we denote $g_y:=\tau_y(g)$. Moreover, let $P:\sd \rightarrow \s$ be linear and continuous. We define $p:\RnRnRn \rightarrow \C$ by
  \begin{align*}
    p (x,\xi,y):=  e^{-ix \cdot \xi}  P \left(e_{\xi} g_y \right)(x) \qquad \text{for all } x,\xi, y \in \Rn.
  \end{align*}
  Then 
  we have for all $\alpha, \beta, \gamma \in \Non$: 
  \begin{align*}
    \p_{\xi}^{\alpha} D_x^{\beta} D_y^{\gamma} p(x,\xi,y) % \\
    = (-1)^{\gamma} \sum_{\beta_1 + \beta_2 = \beta} \binom{\beta}{\beta_1}  e^{-ix \cdot \xi} \left( \ad(-ix)^{\alpha} \ad (D_x)^{\beta_1} P \right) \left(e_{\xi} D_x^{\beta_2 + \gamma} g_y \right)(x).
  \end{align*}
\end{lemma}

\begin{proof}
  Theorem \ref{thm:SchwartzKernel} provides the existence of a Schwartz kernel $f \in \snn$ of $P$.
  Due to $g \in \s$ and $f \in \snn$ we 
  get for 
  all $x \in \Rn$:
  \begin{align*}
    &D_y^{\gamma} \left\{ e^{-ix \cdot \xi} P (e_{\xi} g_y)(x) \right\} =  e^{-ix \cdot \xi}  D_y^{\gamma}  \int f(x,z) e^{iz \cdot \xi} g_y(z) dz \\
    &\qquad = e^{-ix \cdot \xi} \int f(x,z) e^{iz \cdot \xi}  D^{\gamma}_y g_y(z) dz 
     = (-1)^{|\gamma|} e^{-ix \cdot \xi} P (e_{\xi} D^{\gamma}_x g_y)(x).
  \end{align*}
  Inductively with respect to $|\beta|$ we can show for all $\beta,\gamma \in \Non$ and each $x \in \Rn$:
  \begin{align}\label{73}
    & D_x^{\beta} D_y^{\gamma} \left\{ e^{-ix \cdot \xi} P (e_{\xi} g_y)(x) \right\} = (-1)^{|\gamma|} D_x^{\beta} \left\{ e^{-ix \cdot \xi} P (e_{\xi} D^{\gamma}_x g_y)(x) \right\} \notag \\
    & \qquad = (-1)^{|\gamma|} \sum_{\beta_1 + \beta_2 = \beta} \binom{\beta}{\beta_1}  e^{-ix \cdot \xi} \ad (D_x)^{\beta_1} P \left(e_{\xi} D_x^{\beta_2 + \gamma} g_y \right)(x).
  \end{align}
  With Corollary \ref{kor:KernFürAdP} at hand, the iterated commutator $\ad (D_x)^{\beta_1} P$ has a Schwartz kernel $f^{\beta_1} \in \snn$.
  Due to $e_{\xi} D_x^{\beta_2 + \gamma} g_y \in \s$ and $(ix)^{\alpha_2} e_{\xi} D_x^{\beta_2 + \gamma} g_y(x) \in \sindo{x}$ an application of the Leibniz rule and 
%   of (\ref{eqProposition2.66}) 
  interchanginig the derivatives with the integral
  yields for all $x \in \Rn$:
  \begin{align}\label{72}
    &\p_{\xi}^{\alpha} \left\{ e^{-ix \cdot \xi} \ad (D_x)^{\beta_1} P \left(e_{\xi} D_x^{\beta_2 + \gamma} g_y \right)(x) \right\} \notag \\
    & \qquad = \sum_{\alpha_1 + \alpha_2 = \alpha} \binom{\alpha}{\alpha_1} (-ix)^{\alpha_1} e^{-ix \cdot \xi} \int f^{\beta_1} (x,z) (iz)^{\alpha_2} e^{iz \cdot \xi} D_z^{\beta_2 + \gamma} g_y(z) dz \notag \\
    & \qquad = e^{-ix \cdot \xi} (\ad(-ix)^{\alpha} \ad (D_x)^{\beta_1} P) \left( e_{\xi} D_x^{\beta_2 + \gamma} g_y\right) (x).
  \end{align}
  Finally, the combination of (\ref{73}) and (\ref{72}) finishes the proof:
  \begin{align*}
    &\p_{\xi}^{\alpha} D_x^{\beta} D_y^{\gamma} p(x,\xi,y) = \p_{\xi}^{\alpha} D_x^{\beta} D_y^{\gamma} \left\{ e^{-ix \cdot \xi}  P \left(e_{\xi} g_y \right)(x) \right\} \\
    & \qquad = (-1)^{|\gamma|} \sum_{\beta_1 + \beta_2 = \beta} \binom{\beta}{\beta_1} e^{-ix \cdot \xi} (\ad(-ix)^{\alpha} \ad (D_x)^{\beta_1} P) \left( e_{\xi} D_x^{\beta_2 + \gamma} g_y\right) (x)
  \end{align*}
  for all $x, \xi, y \in \Rn$.
\end{proof}

\input{ErweiterungMengeA}

\input{Pseudos}

\input{DoubleSymbols}

%% file: ErweiterungMengeA.tex
\subsection{Extension of the Space of Amplitudes}\label{section:ExtensionSpaceOfAmplitudes}

An important technique for working with smooth pseudodifferential operators are the oscillatory integrals, defined by
\begin{align*}
  \osint e^{-iy \cdot \eta} a(y,\eta) dy \dq \eta := \lim_{\e \rightarrow 0} \iint \chi(\e y, \e \eta) e^{-iy \cdot \eta} a(y,\eta) dy \dq \eta
\end{align*}
for all elements $a$ of the \textit{space of amplitudes} $\mathscr{A}^m_{\tau}(\RnRn)$ $(m,\tau \in \R)$, the set of all smooth functions $a:\Rn \times \Rn \rightarrow \C$ such that
  \begin{align*}
    |\p^{\alpha}_{\eta} \p^{\beta}_y a(y,\eta)| \leq C_{\alpha, \beta} (1+|\eta|)^m (1+|y|)^{\tau}
  \end{align*}
uniformly in $y,\eta \in \Rn$ for all $\alpha, \beta \in \Non$. In order to use the oscillatory integral in the non-smooth case we extend the space of amplitudes 
in the following way:

\begin{Def}\label{Def:ErweiterungRaumAmplituden}
  Let $m, \tau \in \R$ and $N \in \N_0 \cup \{ \infty \}$. We
  define $\mathscr{A}^{m,N}_{\tau}(\RnRn)$
  as the set of all functions $a: \RnRn
\rightarrow \C$ with the following properties: For all
$\alpha, \beta \in \Non$ with $|\alpha| \leq N$ we have
  \begin{enumerate}
    \item[i)] $\p^{\alpha}_{\eta} \p^{\beta}_{y} a(y,\eta) \in
C^0(\RnRnx{y}{\eta})$,
    \item[ii)] $\left|\p^{\alpha}_{\eta} \p^{\beta}_{y} a(y, \eta) \right|
\leq C_{\alpha, \beta} (1 + |\eta|)^m (1 + |y|)^{\tau}$ for all $y, \eta
\in \Rn$.
  \end{enumerate}
\end{Def}
Note that $\mathscr{A}^{m,\infty}_{\tau}(\RnRn) = \mathscr{A}^{m}_{\tau}(\RnRn)$.\\

\begin{bem}\label{bem:eFunktion}
  For $m=2k$, $k \in \N$ we have $e^{ix\cdot\xi} = \<{\xi}^{-m} \<{D_x}^{m} e^{ix\cdot\xi}$.
  Additionally we have for $m=2k+1$, $k \in \N_0$:
  \begin{align*}
    e^{ix\cdot\xi} 
     =\<{\xi}^{-m-1} \<{D_x}^{m-1}e^{ix\cdot\xi}  + \sum_{j=1}^n \<{\xi}^{-m} \frac{\xi_j}{\<{\xi} } \<{D_x}^{m-1} D_{x_j} e^{ix\cdot\xi}
  \end{align*}
\end{bem}

On account of the previous remark, we define for all $m \in \N$
\begin{align*}
  A^m(D_{x},\xi) &:= \<{\xi}^{-m} \<{D_x}^{m} \quad &\text{ if } m \text{ is even},\\
  A^m(D_{x},\xi) &:= \<{\xi}^{-m-1} \<{D_x}^{m-1} -\sum_{j=1}^n \<{\xi}^{-m} \frac{\xi_j}{\<{\xi} } \<{D_x}^{m-1} D_{x_j} \quad &\text{ else}.
\end{align*}

\begin{bem}\label{bem:Absch}
  Let $m, \tau \in \R$ and $N \in \N_0 \cup \{ \infty \}$. Moreover let $\mathscr{B} \subseteq \mathscr{A}^{m,N}_{\tau}(\RnRn)$ be bounded, i.e., for all $\alpha,\beta \in \Non$ with $|\beta| \leq N$ we have
  \begin{align*}
    |\p^{\alpha}_x \p^{\beta}_{\xi} a| \leq C_{\alpha, \beta} \<{\xi}^m \<{x}^{\tau} \qquad \text{for all } a \in \mathscr{B}.
  \end{align*}
  Then we obtain for all $\alpha,\beta \in \Non$ with $|\beta| \leq N$ and $l \in \N$:
  \begin{align*}
    \left| \p^{\alpha}_x \p^{\beta}_{\xi} \left\{ \<{\xi}^{-l-1} \<{D_x}^{l-1} a(x,\xi)-\sum_{j=1}^n \<{\xi}^{-l} \frac{\xi_j}{\<{\xi} } \<{D_x}^{l-1} D_{x_j} a(x,\xi) \right\} \right| 
    \leq C_{\alpha}\<{\xi}^{-l+m} \<{x}^{\tau}
  \end{align*}
  for all $a \in \mathscr{B}$.
\end{bem}

\begin{proof}
  Since $\xi_j \<{\xi}^{-1} \in S^0_{1,0}(\RnRn)$  
  we obtain the claim together with the assumptions and the Leibnitz rule.
\end{proof}

Definition \ref{Def:ErweiterungRaumAmplituden} enables us to extend the definition of the oscillatory integral for functions in the set
$\mathscr{A}^{m,N}_{\tau}(\RnRn)$. It can be proved similarly to e.g. Theorem 3.9 in \cite{PDO} while using Remark \ref{bem:eFunktion} and Remark \ref{bem:Absch}.

\begin{thm}\label{thm:ExistenceOfOscillatoryIntegral}
  Let $m, \tau \in \R$ and $N \in \N_0 \cup \{ \infty\}$ with $N  > n+ \tau$. Moreover, let $\chi \in
\mathcal{S}(\RnRn)$ with $\chi(0,0)=1$ be arbitrary. Then the
\textbf{oscillatory integral}
  \begin{align*}
    \osint e^{-iy \cdot \eta} a(y,\eta) dy \dq \eta := \lim_{\e \rightarrow 0}
\iint \chi(\e y, \e \eta) e^{-iy \cdot \eta} a(y,\eta) dy \dq \eta
  \end{align*}
  exists for each $a \in \mathscr{A}^{m,N}_{\tau}(\RnRn)$. Additionally for all $l,l' \in
\N_0$ with $l > n+m$ and $N \geq l' > n + \tau$ we have
  \begin{align*}
    \osint e^{-iy \cdot \eta} a(y,\eta) dy \dq \eta = \iint e^{-iy \cdot \eta}
A^{l'}(D_{\eta},y) [ A^{l}(D_{y},\eta) a(y,\eta) ]
dy \dq \eta.
  \end{align*}
  Therefore the definition does not depend on the choice of $\chi$.
\end{thm}

Next, we want to convince ourselves that the properties of the oscillatory integral even hold for all functions of the set $\mathscr{A}^{m,N}_{\tau}(\RnRn)$.

\begin{thm}\label{thm:FubiniForOsziInt}\label{thm:VertauschenVonOsziIntUndAbleitungen}
  Let $m,\tau \in \R$ and $k \in  \N$. We define 
   $\tilde{\tau}:= \tau$ if $\tau \geq -k$, $\tilde{\tau}:= -k-0.5$ if $\tau \in \Z$ and $\tau < -k$ and $\tilde{\tau}:= -k-(|\tau| - \lfloor -\tau \rfloor)/2$ else. 
  Moreover, we define $\hat{\tau}:= \tau_+$ if $\tau \geq -k$ and $\hat{\tau}:= \tau- \tilde{\tau}$ else. 
  Additionally let $N \in \N_0 \cup \{ \infty \}$ and  $M:= \max\{ m \in \N_0: N-m \geq l > k+ \tilde{\tau} \text{ for one } l \in \N_0 \}$. Assuming an $a \in \mathscr{A}^{m,N}_{\tau}(\R^{n+k} \times \R^{n+k})$ we define $b: \RnRn \rightarrow \C$ via
  \begin{align*}
    b(y,\eta) := \osint e^{-iy' \cdot \eta'} a(y,y',\eta,\eta') dy' \dq \eta' \qquad \text{for all } y, \eta \in \Rn.
  \end{align*}
  If there is an $\tilde{l} \in \N_0$ with $M \geq \tilde{l} > n + \hat{\tau}$ we obtain
  \begin{align}\label{claim1}
    &\osiint e^{-iy \cdot \eta - iy' \cdot \eta'} a(y,y', \eta, \eta') dy dy' \dq \eta \dq \eta' \notag\\
    & \qquad \qquad = \osint e^{-iy \cdot \eta} \left[ \osint  e^{- iy' \cdot \eta'} a(y,y', \eta, \eta')  dy' \dq \eta' \right] dy \dq \eta. 
  \end{align}
  If there is an $\tilde{l} \in \N_0$ with $N \geq \tilde{l} > k + \tau$ we have $b \in \mathscr{A}^{m_+, M}_{\hat{\tau}}(\R^{n} \times \R^{n})$ and 
  \begin{align}\label{claim2}
    \p_y^{\alpha} \p_{\eta}^{\beta} b(y,\eta) = \osint  e^{-iy' \cdot \eta'}
\p_y^{\alpha} \p_{\eta}^{\beta} a(y,y',\eta,\eta') dy' \dq \eta' \qquad \text{for all } y, \eta \in \Rn
  \end{align}
  for each  $\alpha, \beta \in \Non$ with $|\beta| \leq M$.
\end{thm}

\begin{proof}
  We can show (\ref{claim1}) in the same manner as Theorem 3.13 in \cite{PDO} while using Remark \ref{bem:Absch}. Now let $\tilde{l} \in \N_0$ with $N \geq \tilde{l} > k + \tau$. On account of $\<{(\eta, \eta')}^2 \geq \<{\eta}^2$ for all $\eta, \eta' \in \Rn$ and of Peetre's inequality, cf. \cite[Lemma 3.7]{PDO}, we get 
  \begin{align*}
    \<{(\eta, \eta')}^m \<{(y,y')}^{\tau} \leq C \<{\eta}^{m_+} \<{\eta'}^m \<{y}^{ \hat{\tau} } \<{y'}^{ \tilde{\tau} }  \qquad \text{for all } \eta, \eta', y,y' \in \Rn.	
  \end{align*}
  Hence we obtain for fixed $y,\eta \in \Rn$ and for all $\alpha, \beta \in \Non$, $\tilde{\alpha}, \tilde{\beta} \in \N^k_0 $ with $|\beta| \leq M$ and $|\tilde{\beta}| \leq N-|\beta|$:
  \begin{align*}
    |\p_y^{\tilde{\alpha}} \p_{\eta}^{\tilde{\beta}} \p^{\alpha}_y \p^{\beta}_{\eta} a(y,y',\eta, \eta')| \leq C_{y,\eta} \<{\eta'}^m \<{y'}^{ \tilde{\tau} } \qquad \text{for all } \eta',y' \in \Rn.
  \end{align*}
  Using the previous inequality (\ref{claim2}) can be verified in the same way as \cite[Theorem 3.13]{PDO} while using Remark \ref{bem:Absch}.
\end{proof}

In the same way as Theorem 6.8 in \;\cite{KumanoGo} and Corollary 3.10 in \cite{PDO} we can verify the following statements while using Remark \ref{bem:eFunktion} and Remark \ref{bem:Absch}:

\begin{thm}\label{thm:OscillatoryIntegralGleichung}
 Let $m,\tau \in \R$ and $N \in  \N_0 \cup \{ \infty \}$ with $N > n + \tau$. 
 Moreover, let $l_0, \tilde{l}_0 \in \N_0$ with $\tilde{l}_0 \leq N$. Then
 \begin{align*}
    \osint e^{-iy \cdot \eta} a(y,\eta) dy \dq \eta = \osint e^{-iy \cdot \eta}
 A^{\tilde{l}_0}(D_{\eta},y) A^{l_0}(D_y, \eta) a(y,\eta) 
dy \dq \eta
 \end{align*}
 for every $a \in \mathscr{A}^{m, N}_{\tau}(\R^{n} \times \R^{n})$.
\end{thm}

\begin{kor}\label{kor:VertauschbarkeitVonOsziIntUndLimes}
  Let $m,\tau \in \R$ and $N \in \N_0 \cup \{ \infty \}$ with $N >
n + \tau$. Additionally let $(a_j)_{j \in \N} \subseteq \mathscr{A}^{m, N}_{\tau}(\R^{n} \times \R^{n})$ be bounded, i.e., for all $\alpha, \beta \in \Non$ with $|\alpha| \leq N$: 
  \begin{align*}
    \left| \p^{\beta}_y \p^{\alpha}_{\eta} a_j(y, \eta) \right| \leq C_{\alpha, \beta} \<{\eta}^m \<{y}^{\tau} \qquad \text{for all }y, \eta \in \Rn \text{ and } j \in \N.
  \end{align*}
  Moreover, there is an $a \in \mathscr{A}^{m, N}_{\tau}(\R^{n} \times \R^{n})$ such that
  \begin{align*}
    \lim_{j \rightarrow \infty} \p^{\alpha}_{\eta} \p^{\beta}_y a_j(y,\eta) = \p^{\alpha}_{\eta} \p^{\beta}_y a(y,\eta) \qquad \text{ for all } y, \eta \in \Rn
  \end{align*}
  for each $\alpha, \beta \in \Non$ with $|\alpha| \leq N$. Then
  \begin{align*}
    \lim_{j \rightarrow \infty} \osint  e^{-iy \cdot \eta} a_j(y,\eta) dy \dq \eta = \osint  e^{-iy \cdot \eta} a(y,\eta) dy \dq \eta.
  \end{align*}
\end{kor}

\begin{thm}\label{thm:VariableTransformationOfOsiInt}
  Let $m,\tau \in \R$ and $N \in  \N_0 \cup \{ \infty \}$ with $N  >
n + \tau$. For $a \in \mathscr{A}^{m, N}_{\tau}(\R^{n} \times \R^{n})$ we have:
  \begin{align*}
    \osint  e^{-i(y + y_0) \cdot (\eta + \eta_0)} a(y + y_0,\eta + \eta_0) dy \dq \eta = \osint  e^{-iy \cdot \eta} a(y,\eta) dy \dq \eta.
  \end{align*}
\end{thm}

%% file: Pseudos.tex
\section{Pseudodifferential Operators} \label{Kapitel: PDO}
\subsection{Properties of Pseudodifferential Operators} \label{Section:PropertiesOfPDO}

For derivatives of non-smooth symbols we are able to verify the next statement:

\begin{lemma}\label{lemma:boundOfSymbolsOnCs}
  Let $m\in \N_0$, $M \in \N_0 \cup \{ \infty \}$ and $0<s\leq 1$. Additionally let $\mathscr{B} \subseteq C^{m,s}S^0_{0,0}(\RnRn; M)$ be a bounded subset. Considering $\alpha, \gamma \in \N_0^n$ with $|\gamma| \leq M-1$ and $|\alpha| \leq m$,  the set $\{ \p^{\alpha}_x \pa{\gamma} a : a \in \mathscr{B} \} \subseteq C^{0,s}(\RnRn)$ is bounded.
\end{lemma}

\begin{proof}
   First of all we choose arbitrary $\alpha, \gamma \in \N_0^n$ with $|\gamma| \leq M-1$ and $|\alpha| \leq m$. Since $\mathscr{B} \subseteq C^{m,s}S^0_{0,0}(\RnRn; M)$ is a bounded subset,
  we get that
  \begin{align}
    \hspace{-3mm}&\{ \p^{\alpha}_x \pa{ \gamma} a : a \in \mathscr{B} \} \subseteq C^0_b(\RnRn) \text{ is bounded and}, \label{20}
\\
    \hspace{-3mm}&\sup_{(x,\xi)\neq (y,\eta)} \hspace{-3mm}\frac{| \p^{\alpha}_x \pa{ \gamma} a(x,\eta)- \p^{\alpha}_x \pa{ \gamma} a(y,\eta)|}{|(x,\xi) - (y,\eta)|^{s}} 
    \leq  \sup_{\eta} \| \pa{ \gamma} a(.,\eta) \|_{C^{m,s}(\Rn)} < C_{\gamma} \hspace{1mm} \forall a \in \mathscr{B}. \label{21}
  \end{align}
  On account of the fundamental theorem of calculus in the case $|\xi-\eta|<1$ with $\xi \neq \eta$ and because of (\ref{20}) for $|\xi-\eta|\geq 1$ we can show
  \begin{align}\label{24}
    \sup_{(x,\xi)\neq (y,\eta)} \frac{| \p^{\alpha}_x \pa{ \gamma} a(x,\xi)- \p^{\alpha}_x \pa{\gamma} a(x,\eta)|}{|(x,\xi) - (y,\eta)|^{s}}   
    \leq C_{\gamma} \qquad \text{for all } a \in \mathscr{B}.
  \end{align}
  Collecting the estimates (\ref{20}),(\ref{21}) and (\ref{24}) we finally obtain the claim.
\end{proof}

Now let us mention some boundedness results for pseudodifferential operators needed later on. 

\begin{thm} \label{thm:stetigAufS}
  Let $p \in \Sn{m}{1}{0}$ with $m \in \R$. Then 
  \begin{align*}
    p(x,D_x): \s \rightarrow \s
  \end{align*}
  is a bounded mapping. More precisely, for every $k \in \N_0$ we can show 
  \begin{align*}
    |p(x,D_x)f|_{k,\mathcal{S}} \leq C_k |p|^{(m)}_k |f|_{\tilde{m}, \mathcal{S}} \qquad \textrm{for all } f \in \s,
  \end{align*}
  where 
  $\tilde{m}:= \max \{ 0, m + 2(n+1) + k\}$ if $m \in \Z$ and $\tilde{m}:= \max \{ 0, \lfloor m \rfloor + 2n+3 + k\}$ else.
\end{thm}

We refer to e.\,g.\,\cite[Theorem 3.6]{PDO} for the proof. Non-smooth pseudodifferential operators
with coefficients in a Banach space $X$ with $C^{\infty}_c(\Rn) \subseteq X \subseteq C^0(\Rn)$, see e.\,g.\,\cite{Taylor2} for the definition,
have similar properties if the next estimate holds for some $N \in \N$ and $C_{\tilde{m},\tau}>0$:
\begin{align}\label{eqVonBemerkung}
    \|e_{\xi} \cdot a(.,\xi) \|_{X} \leq C_{\tilde{m},\tau} \<{\xi}^{ N } \| a(.,\xi) \|_{X} \qquad \text{for all } \xi \in \Rn,
\end{align}
where $a$ denotes the symbol of the pseudodifferential operator. In the next remark we mention some spaces, where the previous estimate is fulfilled:

\begin{bem}\label{bem:AbschNormMitEFunktion}
  Let $X \in \{ C^{\tilde{m},\tau}, C^{\tilde{m} + \tau}_{\ast}, H^{\tilde{m}}_q, W^{\tilde{m}, q}_{uloc} \}$ with $\tilde{m} \in \N_0$, $0< \tau \leq 1$ and $1<q < \infty$. Additionally we assume $\delta=0$ in the case $X \notin \{  C^{\tilde{m},\tau}, C^{\tilde{m} + \tau}_{\ast} \}$ and $\tilde{m}>n/q$ if $X \in \{ H^{\tilde{m}}_q, W^{\tilde{m}, q}_{uloc} \}$. For $0 \leq \rho, \delta \leq 1$ and $M \in \N_0 \cup \{ \infty \}$ we choose an arbitrary $a \in X \Snn{m}{\rho}{\delta}{M}$. Then inequality (\ref{eqVonBemerkung}) holds for  $N = \tilde{m} +2$ in the case $X=C^{\tilde{m} + \tau}_{\ast}$, for $N= \tilde{m}+1$ in the case $X=C^{\tilde{m}, \tau}$ and for $N= \tilde{m}$ else. 
  We refer to \cite[Definition 4.11]{Diss} for the definition of the uniformly local Sobolev Spaces $W^{\tilde{m}, q}_{uloc} $.
\end{bem}

\begin{proof}
  For $X \in \{ H^{\tilde{m}}_q, W^{\tilde{m}, q}_{uloc} \}$ the claim can be verified by using the definition of these spaces and the Leibniz rule.
  With the multiplication property
  \begin{align*}
    \|fg \|_{C^s_{\ast}} \leq C \|f\|_{C^s_{\ast}} \|g\|_{C^s_{\ast}} \qquad \text{for all } f,g \in C^s_{\ast}(\Rn),
  \end{align*}
  and the embedding $C_b^{\tilde{m} +\lfloor \tau \rfloor +1 }(\Rn) \hookrightarrow C_{\ast}^{\tilde{m} +\tau}(\Rn)$ at hand, we are in the position to prove the remark for $X=C^{\tilde{m},\tau}_{\ast}$:
  \begin{align*}
    \|e_{\xi} \cdot a(.,\xi) \|_{C^{\tilde{m},\tau}_{\ast}}  
    \leq C_{\tilde{m},\tau} \| e_{\xi}  \|_{ C_b^{\tilde{m} +\lfloor \tau \rfloor +1 } } \| a(.,\xi) \|_{C^{\tilde{m},\tau}_{\ast}} 
    \leq C_{\tilde{m},\tau} \<{\xi}^{ \tilde{m}+2 } \| a(.,\xi) \|_{C^{\tilde{m},\tau}_{\ast}}
  \end{align*}
  for all $\xi \in \Rn$. It remains to prove the case $X=C^{\tilde{m},\tau}$. Using the mean value theorem in the case $|x_1-x_2| \leq 1$, $x_1 \neq x_2$ we obtain
  \begin{align*}
    \max_{x_1 \neq x_2} \frac{|e^{ix_1 \cdot \xi} - e^{ix_2 \cdot \xi}|}{|x_1-x_2|^\tau} \leq 2 \<{\xi} \qquad \text{for all } \xi \in \Rn.
  \end{align*} 
    On account of the previous inequality we are able to verify the next estimate:  
  \begin{align}\label{eq59}
    \max_{|\alpha| \leq \tilde{m}} \sup_{x_1 \neq x_2} \frac{|e^{ix_1 \cdot \xi} \p_x^{\alpha} a(x_1, \xi) -   e^{ix_2 \cdot \xi}  \p_x^{\alpha} a(x_2, \xi) |}{|x_1-x_2|^\tau} 
    \leq  C_{\tilde{m},\tau} \<{\xi} \|a(.,\xi)\|_{C^{\tilde{m},\tau}_{\ast}}
  \end{align}
  for all $ \xi \in \Rn$.  
  Moreover we are able to show
  \begin{align}\label{eq60}
    \|e_{\xi} \cdot a(.,\xi) \|_{C^{\tilde{m}}_{b}} \leq C_{\tilde{m}} \<{\xi}^{ \tilde{m} } \| a(.,\xi) \|_{C^{\tilde{m},\tau} } \qquad  \text{for all } \xi \in \Rn.
  \end{align}
  A combination of the inequalities (\ref{eq59}) and (\ref{eq60}) yields 
  \begin{align*}
    \|e_{\xi} \cdot a(.,\xi) \|_{C^{\tilde{m},\tau}} \leq C_{\tilde{m},\tau} \<{\xi}^{ \tilde{m}+1 } \| a(.,\xi) \|_{C^{\tilde{m},\tau}} \qquad \text{for all } \xi \in \Rn.
  \end{align*}
  \vspace*{-1cm}

\end{proof}

The previous remark enables us to prove the next boundedness result:

\begin{lemma}\label{lemma:stetigkeitInS}
  We consider a Banach space $X$ with $C^{\infty}_c(\Rn) \subseteq X \subseteq C^0(\Rn)$ such that inequality (\ref{eqVonBemerkung}) holds. Let $m\in \R$, $M \in \N_0 \cup \{ \infty \}$ and $\delta=0$ in the case $X \notin \{  C^{\tilde{m},\tau}, C^{\tilde{m} + \tau}_{\ast} \}$ and $0 \leq \rho,\delta \leq 1$ else. Assuming $p \in X\Sallgn{m}{\rho}{\delta}{n}{n}{M}$,  we obtain the continuity of $p(x, D_x): \s \rightarrow X$. 
\end{lemma}

\begin{proof}
  Let $u \in \s$ be arbitrary. An application of $p \in X\Sallgn{m}{\rho}{\delta}{n}{n}{M}$, $u \in \s$ and Remark \ref{bem:AbschNormMitEFunktion} 
  yields
  \begin{align*}
    \| p(x, D_x) u(x) \|_X 
    &\leq \int \|e_{\xi} p(.,\xi) \|_X |\hat{u}(\xi)| \dq \xi 
    \leq C \int \<{\xi}^{-(n+1)} \dq \xi |\hat{u}|_{\hat{m} +(n+1), \mathcal{S} } \\
    &\leq C |u|_{\hat{m} +2(n+1), \mathcal{S} } \qquad \text{for all } x \in \Rn \text{ and some }\hat{m} \in \N.
  \end{align*}
  \vspace*{-1cm}

\end{proof}

In the case $X=C^{\tilde{m},\tau}$ this statement was already proven in \cite[Theorem 3.6]{Koeppl}. For a bounded subset of $\mathscr{B} \subseteq X\Sallgn{m}{\rho}{\delta}{n}{n}{M}$, where $X, m, \rho,\delta$ and $M$ are defined as in the previous lemma, we  are even able to improve the statement of Lemma \ref{lemma:stetigkeitInS}: Verifying the proof of Lemma \ref{lemma:stetigkeitInS} yields the boundedness of
\begin{align}\label{eq23}
  \left\{ p(x, D_x) : p \in \mathscr{B} \right\} \subseteq \mathscr{L}(\s; X).
\end{align}
In the literature such problems are mostly not investigated. Usually just boundedness results are shown in different cases. Verifying these proofs in order to get similar results as (\ref{eq23}) is often very complex. With the next lemma at hand, such problems are much easier to prove.  

\begin{lemma}\label{lemma:UnabhangigkeitVomSymbol}
  Let $N \in \N_0 \cup \{ \infty \}$ and $X$ be a Banach space with $C^{\infty}_c(\Rn) \subseteq X \subseteq C^0(\Rn)$. Additionally let $m$, $\rho$ and $\delta$ be as in the last lemma.
  We consider that $\mathscr{B}$ is the topological vector space $S^m_{\rho, \delta} (\RnRn)$ or $X S^m_{\rho, \delta} (\RnRn;N)$. In the case $\mathscr{B}=S^m_{\rho, \delta} (\RnRn)$ we set $N:=\infty$. Moreover, let $X_1,X_2$ be two Banach spaces with the following properties:
  \begin{enumerate}
    \item[i)] $\s \subseteq X_1, X_2 \subseteq \sd$,
    \item[ii)] $\s$ is dense in $X_1$ and in $ X'_2$,
    \item[iii)] $a(x,D_x) \in \mathscr{L}(X_1, X_2)$ for all $a \in \mathscr{B}$. 
  \end{enumerate}
  Then there is a $k \in \N$ with $k \leq N$ such that
  \begin{align*}
    \| a(x, D_x)f \|_{\mathscr{L}(X_1;X_2)} \leq C |a|^{(m)}_{k} \qquad \text{for all } a \in \mathscr{B}.
  \end{align*}
\end{lemma}

% Beweis ist auf Seite 1/2 vom 21.04.2015 aufgeschrieben
\begin{proof}
  First of all we define for $f,g \in \s$ with $\|f\|_{X_1} \leq 1$ and $\|g\|_{X'_2} \leq 1$ the operator $\op_{f,g} : \mathscr{B} \rightarrow \C$ by $\op_{f,g}(a):= \skh{a(x, D_x)f}{g}{X_2, X'_2}$. Using $iii)$ we get the existence of a constant $C$, independent of $f,g \in \s$ with $\|f\|_{X_1} \leq 1$ and $\|g\|_{X'_2} \leq 1$, such that 
  \begin{align*}
    |\skh{a(x, D_x)f}{g}{X_2, X'_2}| 
%     &\leq \left\| a(x, D_x)f \right\|_{X_2} \left\| g \right\|_{X_2'}
    \leq C \left\| a(x, D_x)\right\|_{\mathscr{L}(X_1; X_2) } \left\| f \right\|_{X_1} \left\| g \right\|_{X_2'} 
    \leq C \left\| a(x, D_x)\right\|_{\mathscr{L}(X_1; X_2) }.
  \end{align*}
  Consequently the set
  \begin{align*}
    \left\{ \op(a)_{f,g} : f,g \in \s \text{ with } \|f\|_{X_1} \leq 1 \text{ and } \|g\|_{X'_2} \leq 1  \right\} \subseteq \C
  \end{align*}
  is bounded for each $a \in \mathscr{B}$.
  An application of the theorem of Banach-Steinhaus, cf.\;e.g.\;\cite[Theorem 2.5]{Rudin} provides that
  \begin{align*}
    \left\{ \op_{f,g} : f,g \in \s \text{ with } \|f\|_{X_1} \leq 1 \text{ and } \|g\|_{X'_2} \leq 1  \right\}
  \end{align*}
  is equicontinuous. With the equicontinuity of the previous set at hand, we get the existence of a $k \in \N$ with $k \leq N$ and a constant $C>0$ such that
  \begin{align*}
     |\op_{f,g}(a)| \leq C | a |^{(m)}_{k} \quad \text{for all }  a \in \mathscr{B}, f,g \in \s \text{ with } \|f\|_{X_1} \leq 1, \|g\|_{X'_2} \leq 1.
  \end{align*}
  Since $\s$ is dense in $X_1$ and in $ X'_2$, the previous inequality even holds for all $f \in X_1$ and $g \in X_2'$ with $\|f\|_{X_1} \leq 1$ and $\|g\|_{X'_2} \leq 1$. This implies the claim:
  \begin{align*}
    \| a(x, D_x) \|_{\mathscr{L}(X_1;X_2)} 
    &= \sup_{\|f\|_{X_1} \leq 1} \| a(x, D_x)f \|_{X_2}
    = \sup_{\|f\|_{X_1} \leq 1} \sup_{\|g\|_{X_2'} \leq 1} |\op_{f,g}(a)| 
    \leq C |a|^{(m)}_{k}
  \end{align*}
  for all $a \in \mathscr{B}$.
\end{proof}

Next we summarize boundedness results for pseudodifferential operators as maps between two Bessel potential spaces. In the smooth case we refer 
e.g.\;to \cite[Theorem 5.20]{PDO}.

\begin{thm}\label{thm:stetigInBesselPotRaum}
  Let $m\in \R$, $p \in \Sn{m}{1}{0}$ and $1<q<\infty$. Then $p(x,D_x)$ extends to a bounded linear operator
  \begin{eqnarray*}
    p(x,D_x): H_q^{s+m}(\Rn) \rightarrow H^s_q(\Rn) \qquad \textrm{for all } s \in \R.
  \end{eqnarray*} 
\end{thm}

\begin{thm} \label{thm:stetigInHoelderRaum}
  Let $m \in \R$, $0 \leq \delta \leq \rho \leq 1$ with $\rho > 0$ and $1 < p < \infty $. Additionally let $\tau > \frac{1- \rho}{1-\delta} \cdot \frac{n}{2}$ if $\rho <1$ and $\tau>0$ if $\rho =1$ respectively. Moreover, let $N \in \N \cup \{ \infty \}$ with $N> n/2$ for $2 \leq p < \infty$ and $N>n/p$ else. Denoting $k_p := (1- \rho)n \left| 1/2 - 1/p \right|$, let $\mathscr{B} \subseteq C^{\tau}_{\ast} \Snn{m- k_p}{\rho}{\delta}{N}$ be a bounded subset. Then for each real number $s$ with the property
  $$(1-\rho)\frac{n}{p}-(1-\delta)\tau < s < \tau$$
  there is a constant $C_s>0$, independent of $a \in \mathscr{B}$, such that
  \begin{align*}
    \| a(x, D_x)f \|_{ H_p^{s} } \leq C_s \| f \|_{ H_p^{s+m} } \qquad \text{for all } f \in H_p^{s+m}(\Rn) \text{ and } a \in \mathscr{B}.
  \end{align*}
\end{thm}

\begin{proof}
  In the case $ 2 \leq p <\infty$ the theorem was shown in \cite[Theorem 2.7]{Marschall} for $\sharp \mathscr{B} = 1$. The case $1<p<2$ has been proved in \cite[Theorem 4.2]{Marschall} for $\sharp \mathscr{B} = 1$. Thus it remains to verify whether the constant $C_s$ is independent of $a \in \mathscr{B}$. We define $p'$ by $1/p + 1/p' = 1$. Since $\s$ is dense in $H^{s+m}_p(\Rn)$ and $H^{-s}_{p'}(\Rn)$,
  the theorem holds because of Lemma \ref{lemma:UnabhangigkeitVomSymbol}.
\end{proof}

In the case $\sharp \mathscr{B} = 1$, the previous theorem also holds for $p = 1$ or $p=\infty$, cf.\; \cite[Theorem 2.7 and Theorem 4.2 ]{Marschall}. \\

On account of Theorem 2.1 in \cite{Marschall} and  Lemma 2.9 in \cite{Marschall} the next boundedness results hold:

\begin{thm} \label{thm:1stetigInHoelderRaum00}
  Let $m \in \R$ and $\tau > \frac{n}{2}$. Moreover, let $N \in \N \cup \{ \infty \}$ with $N> n/2$. Additionally let $a \in C^{\tau}_{\ast} \Snn{m}{0}{0}{N}$. Then for each real number $s \in \left( \frac{n}{2}-\tau , \tau \right)$
  there is a constant $C_s>0$ such that
  \begin{align*}
    \| a(x, D_x)f \|_{ H_2^{s} } \leq C_s \| f \|_{ H_2^{s+m} } \qquad \text{for all } f \in H_2^{s+m}(\Rn).
  \end{align*}
\end{thm}

\begin{thm} \label{thm:stetigInHoelderRaum00}
  Let $m \in \R$, $N> n/2$ and $\tau > 0$. Moreover let $P$ be an element of $ \op C^{\tau}_{\ast} \Snn{m-n/2}{0}{0}{N}$. Then the operator
  \begin{align*}
    P: H_2^{s+m}(\Rn) \rightarrow H_2^s(\Rn) \qquad \textrm{is continuous for all } -\tau < s < \tau.
  \end{align*}
\end{thm}

\begin{lemma}\label{lemma:AbschVonBeschrPODMenge}
  Let $s \in \R^+$ with $s \notin \N$, $m \in \R$ and $0 \leq \rho, \delta \leq 1$. Additionally let $M \in \N_0 \cup \{ \infty \}$. Moreover, $\mathscr{B} \subseteq C^s \Snn{m}{\rho}{\delta}{M}$ should be a bounded subset and $u \in \s$. For every $N \in \N_0$ with $2N \leq M$ we have
  \begin{align*}
    |a(x,D_x)u(x)| \leq C_{N,n} \<{x}^{-2N} \qquad \text{for all } x \in \Rn \text{ and } a \in \mathscr{B}.
  \end{align*}
  Note  that $C_{N,n}$ is dependent on $u \in \s$.
\end{lemma}

\begin{proof}
  Let $N \in \N_0$ with $2N \leq M$. 
  Choosing $M_{m,n} \in \N$ with $-M_{m,n} < -n -|m|$, we get for all $a \in \mathscr{B}$ and all $x \in \Rn$
  by means of $u \in \s$
  and the boundedness of $\mathscr{B} \subseteq C^s \Snn{m}{\rho}{\delta}{M}$:
  \begin{align}\label{p2}
    \left| \<{ D_{\xi} }^{2N} \left[ a(x,\xi) \hat{u}(\xi) \right] \right| 
    \leq C_{N,n} \<{\xi}^{m -M_{m,n}} 
    \in L^1(\Rn_{\xi}).
  \end{align}
  Here $C_{N,n}$ is independent of $x,\xi \in \Rn$ and $a \in \mathscr{B}$.  
    On account of 
    (\ref{p2}) and integration by parts with respect to $\xi$
  we conclude the claim:
  \begin{align*}
    \left| \<{x}^{2N} a(x,D_x)u(x) \right| 
     = \left| \int e^{ix \cdot \xi} \<{ D_{\xi} }^{2N} \left[ a(x,\xi) \hat{u}(\xi) \right]  \dq \xi \right| 
    \leq C_{N,n}
  \end{align*}
  for all $a \in \mathscr{B}$ and $x \in \Rn$.
\end{proof}

%% file: DoubleSymbols.tex
\subsection{Double Symbols}\label{Section:DoubleSymbols}

\begin{Def}
	Let $0<s \leq 1$, $m \in \N_0$ and $\tilde{m},m'\in \R$. Furthermore, let $N \in \N_0 \cup \{ \infty \}$ and $0 \leq \rho \leq 1$. Then the space of \textit{non-smooth double (pseudodifferential) symbols} $C^{m,s} S^{\tilde{m},m'}_{\rho, 0}(\RnRn \times \RnRn;N)$ 
	is the set of all functions $p:\R^n_x \times\R^n_{\xi} \times \R^n_{x'} \times\R^n_{\xi'} \rightarrow \C$ such that
	\begin{itemize}
		\item[i)] $\pa{\alpha} \p^{\beta'}_{x'} \p_{\xi'}^{\alpha'} p \in C^s(\R^n_x)$ and $\p_x^{\beta} \pa{\alpha} \p^{\beta'}_{x'} \p_{\xi'}^{\alpha'} p \in C^{0}(\R^n_x \times \R^n_{\xi} \times \R^n_{x'} \times\R^n_{\xi'})$,
% 		\item[ii)] $\pa{\alpha}  \p^{\beta'}_{x'} \p_{\xi'}^{\alpha'} p(.,\xi, x', \xi') \in C^s(\R^n)$,
		\item[ii)] $\| \pa{\alpha}  \p^{\beta'}_{x'} \p_{\xi'}^{\alpha'} p(.,\xi, x', \xi')  \|_{C^s(\R^n)} \leq C_{\alpha, \beta', \alpha'} \<{\xi}^{\tilde{m}-\rho|\alpha|} \<{\xi'}^{m'-\rho|\alpha'|}$
	\end{itemize}
  for all $\xi, x', \xi' \in \Rn$ and arbitrary $\beta, \alpha, \beta', \alpha' \in \N_0^n$ with $|\beta| \leq m$ and $|\alpha| \leq N$. 
  Here the constant $C_{\alpha, \beta', \alpha'}$ is independent of $\xi, x', \xi' \in \Rn$. 
  In the case $N= \infty$ we write $C^{m,s} S^{\tilde{m},m'}_{\rho, 0}(\RnRn \times \RnRn)$
  instead of $C^{m,s} S^{\tilde{m},m'}_{\rho, 0}(\RnRn \times \RnRn;\infty)$. Furthermore, we define the set of semi-norms $\{|.|^{\tilde{m},m'}_k : k \in \N_0 \}$ by
  \begin{align*}
    |p|^{\tilde{m},m'}_k \hspace{-1.2mm}:= \hspace{-0.2mm} \max_{\substack{ |\alpha| + |\beta'| + |\alpha'| \leq k \\ |\alpha| \leq N} } \sup_{\xi, x', \xi' \in \Rn} \hspace{-2mm} \| \pa{\alpha}  \p^{\beta'}_{x'} \p_{\xi'}^{\alpha'} p(.,\xi, x', \xi')  \|_{C^{m,s}(\R^n)} \<{\xi}^{-(\tilde{m}-\rho|\alpha|)} \<{\xi'}^{-(m'-\rho|\alpha'|)}.
  \end{align*}
\end{Def}

Due to the previous definition, $p \in C^{m,s}\Sallg{\tilde{m}}{\rho}{\delta}{n}{N}$ is often called a \textit{non-smooth single symbol}. \\

The associated operator of a non-smooth double symbol is defined in the following way:

\begin{Def}
  Let $0<s \leq 1$, $m \in \N_0$, $0 \leq \rho \leq 1$ and $\tilde{m},m'\in \R$. Additionally let $N \in \N_0 \cup \{ \infty \}$. Assuming $ p \in C^{m,s} S^{\tilde{m},m'}_{\rho, 0}(\RnRn \times \RnRn; N) $,  we define the pseudodifferential operator $P = p(x,D_x, x', D_{x'})$ 
  such that for all $u \in \s$ and $x \in \Rn$
  \begin{align*}
    P u(x) := \osiint e^{-i(y \cdot \xi + y' \xi')} p(x,\xi,x+y,\xi') u(x+y+y')dy dy' \dq \xi \dq \xi' .
  \end{align*}
  Note, that we can verify the existence of the previous oscillatory integral by using the properties of such integrals. For more details, see \cite[Lemma 4.64]{Diss}.
\end{Def}

The set of all non-smooth pseudodifferential operators whose double symbols are in the symbol-class $ C^{m,s} S^{\tilde{m},m'}_{\rho, 0}(\RnRn \times \RnRn; N) $ is denoted by 
  $$\op C^{m,s} S^{\tilde{m},m'}_{\rho, 0}(\RnRn \times \RnRn; N). $$

For later purposes we will need a special subset of the non-smooth double symbols $C^{m,s} S^{\tilde{m},0}_{\rho, 0}(\RnRn \times \RnRn;N)$: For $0<s \leq 1$, $m \in \N_0$, $N \in \N_0 \cup \{ \infty \}$ and $\tilde{m} \in \R$ we denote the space $C^{m,s} S^{\tilde{m}}_{\rho, 0}(\RnRn \times \Rn;N)$ 
as the set of all non-smooth symbols $p \in C^{m,s} S^{\tilde{m},0}_{\rho, 0}(\RnRn \times \RnRn;N)$ which are independent of $\xi'$.
Then we define the pseudodifferential operator $p(x,D_x,x')$
by $$p(x,D_x,x') := p(x,D_x, x', D_{x'}).$$
 The set of all non-smooth pseudodifferential operators whose double symbols are in $C^{m,s} S^{\tilde{m}}_{\rho, 0}(\RnRn \times \Rn;N)$ is denoted by $\op C^{m,s} S^{\tilde{m}}_{\rho, 0}(\RnRn \times \Rn;N)$.\\

Pseudodifferential operators of the symbol-class $C^{m,s} S^{\tilde{m}}_{\rho, 0}(\RnRn \times \Rn;N)$ applied on a Schwartz function can be presented in the following way:

\begin{lemma}\label{lemma:PseudosMitDoppelsymbol}
  Let $0<s<1$, $\tilde{m} \in \N_0$, $0 \leq \rho \leq 1$, $m\in \R$ and $N \in \N_0 \cup \{ \infty \}$. Considering $a \in C^{\tilde{m},s} S^{m}_{\rho, 0}(\RnRn \times \Rn;N)$, we obtain for all $u \in \s$: 
  \begin{align*}
    a(x,D_x,x') u(x) = \osint e^{i(x-y)\cdot \xi} a(x,\xi, y) u(y) dy \dq \xi  \qquad \text{for all } x\in \Rn.
  \end{align*}
\end{lemma}

\begin{proof}
  Let $u \in \s$ and $x \in \Rn$ be arbitrary. Then
  \begin{align*}
    \<{y'}^{-2l} \<{D_{\xi'}}^{2l} a(x,\xi, z) u(z+y') \in \mathscr{A}_{-2n-2}^{m_+, N}(\R^{2n}_{(z,y')} \times \R^{2n}_{(\xi,\xi')}).
  \end{align*} 
  With Theorem \ref{thm:ExistenceOfOscillatoryIntegral}, Corollary \ref{kor:VertauschbarkeitVonOsziIntUndLimes},
  Theorem \ref{thm:VariableTransformationOfOsiInt} and Theorem \ref{thm:FubiniForOsziInt} at hand, we get
  \begin{align*}
    &a(x,D_x,x')u(x) \\
    &\quad = \osiint e^{-i(y \cdot \xi + y' \cdot \xi')} \<{y'}^{-2l} \<{D_{\xi'}}^{2l} [ a(x,\xi,x+y) u(x+y+y')]dy dy' \dq \xi \dq \xi' \\
    &\quad = \osiint e^{-i(z-x) \cdot \xi } e^{ -iy' \cdot \xi'} a(x,\xi,z) \<{y'}^{-2l} \<{D_{\xi'}}^{2l} u(z+y')dz dy' \dq \xi \dq \xi'\\
    &\quad = \osint e^{i(x-z) \cdot \xi } a(x,\xi,z) \left[ \osint e^{ -iy' \cdot \xi'} \<{y'}^{-2l} \<{D_{\xi'}}^{2l} u(z+y')dy' \dq \xi' \right]  \hspace{-0.1cm} dz \dq \xi
  \end{align*} 
  By means of
  \begin{align*}
    \<{y'}^{-2l} \<{D_{\xi'}}^{2l} u(z+y') \in \sindo{y'} \subseteq \mathscr{A}^0_{-k}(\RnRnx{y'}{\xi'})
  \end{align*}
  we are able to apply Theorem \ref{thm:OscillatoryIntegralGleichung} and Theorem \ref{thm:VariableTransformationOfOsiInt} and get
  \begin{align*}
    \osint e^{ -iy' \cdot \xi'} \<{y'}^{-2l} \<{D_{\xi'}}^{2l} u(z+y')dy' \dq \xi' 
    = \osint e^{ -i(\tilde{z}-z) \cdot \xi'}  u(\tilde{z}) d\tilde{z} \dq \xi'
    = u(z).
  \end{align*}
  For the proof of the last equality we refer to \cite[Example 3.11]{PDO}.
  Combining all these results we conclude the proof. 
\end{proof}

\begin{bem}\label{bem:AbleitungVonPDOWiederPDO}
  Let $0<s<1$, $m \in \R$, $\tilde{m} \in \N_0$ and $N \in \N_0 \cup \{ \infty \}$. The boundedness of the subset $\mathscr{B} \subseteq C^{\tilde{m},s} S^{m}_{\rho, 0}(\RnRn \times \Rn; N)$, $0\leq \rho \leq 1$, implies the boundedness of 
  $$\left\{ \p^{\delta}_x \pa{\gamma} a : a \in \mathscr{B} \right\} \subseteq C^{\tilde{m}- |\delta|,s} S^{m-\rho|\gamma|}_{\rho, 0}(\RnRn \times \Rn; N-|\gamma|)$$ 
  for each $\gamma, \delta \in \Non$ with $|\delta| \leq \tilde{m}$ and $|\gamma| \leq N$.
\end{bem}

  \begin{proof}
    The claim is a direct consequence of the definition of the double symbols.
  \end{proof}

%% file: Characterization.tex
\section{Characterization of Non-Smooth Pseudodifferential Operators} \label{Section:Charakterisierung}

Throughout the whole section $( \varphi_j)_{ j \in \N_0}$ is an arbitrary but fixed dyadic partition of unity on $\R^n$, that is a partition of unity with 
\begin{align*}%\label{eq22}
  \supp \varphi_0 \subseteq \overline{B_2(0)} \qquad \text{and} \qquad \supp \varphi_j \subseteq \{ \xi \in \Rn: 2^{j-1} \leq |\xi| \leq 2^{j+1}\} 
\end{align*}
for all $ j \in \N$. Moreover we define for every $m \in \R$ the order reducing pseudodifferential operator $\Lambda^m:=\lambda^m(D_x)$, where $\lambda^m(\xi):=\<{\xi}^m \in S^{m}_{1,0}(\RnRn)$.

\subsection{Pointwise Convergence in $C^{m,s} S^0_{0,0}$}\label{section:pointwiseConvergence}

For a bounded sequence $(p_{\e})_{\e >0} \subseteq C^{m,s}S^0_{0,0}(\RnRn;M)$, we show the existence of a subsequence of $(p_{\e})_{\e >0}$ which converges pointwise in $C^{m,s}S^0_{0,0}(\RnRn;M-1)$. 
To reach this goal we need the next lemma:

\begin{lemma}\label{lemma:convergence}
   Let $m \in \N_0$, $0<s\leq 1$ 
%    and $( \Omega_j )_{j \in \N} $ be a countable open cover of bounded sets with Lipschitz boundary of $\Rn$. Additionally let 
   and $(p_{\e} )_{\e>0} \subseteq C^{m,s}(\Rn)$ be a bounded sequence. Then there is a subsequence $(p_{\e_k} )_{k \in \N} \subseteq (p_{\e} )_{\e>0}$ with $\e_k \rightarrow 0$ for $k \rightarrow \infty$ and a $p \in C^{m,s}(\Rn)$ such that for all $\beta \in \N_0^n$ with $|\beta| \leq m$
  \begin{align*}
    \p^{\beta}_x  p_{\e_k} \xrightarrow[]{k \rightarrow \infty} \p^{\beta}_x p
  \end{align*}
  converges uniformly on each compact set $K \subseteq \Rn$.  
%   $\bar{\Omega}_j$, $j \in \N$.
\end{lemma}

\begin{proof}
  It is sufficient to prove the claim for each $\overline{B_j(0)}$, $j \in \N$.
  Due to the boundedness of $( p_{\e}|_{\overline{B_j(0)} })_{\e > 0} \subseteq C^{m,s}(\overline{B_j(0)})$ and the compactness of the embedding 
  $C^{m,s}(\overline{B_j(0)}) \subseteq C^{m}(\overline{B_j(0)})$ we get by a diagonal sequence argument the existence of a subsequence $(p_{\e_k} )_{k \in \N} \subseteq (p_{\e} )_{\e>0}$ with $\e_k \rightarrow 0$ for $k \rightarrow \infty$ and of unique functions $p_{B_j(0)} \in C^m(\overline{B_j(0)})$ such that
  \begin{align*}
    p_{\e_k} \xrightarrow[]{k \rightarrow \infty} p_{B_j(0)} \qquad \text{in } C^m(\overline{B_j(0)}) \text{ for all } j \in \N.
  \end{align*}
  We define $p:\Rn \rightarrow \C$ via
    $p(x) := p_{B_j(0)}(x)$ for all $x \in \overline{B_j(0)}$ and each $j \in \N$.
  This implies the uniform convergence of
  \begin{align*}
    \p^{\beta}_x  p_{\e_k} \xrightarrow[]{k \rightarrow \infty} \p^{\beta}_x p \qquad \text{on } \overline{B_j(0)}
  \end{align*}
  for all $j \in \N$ and $\beta \in \Non$ with $|\beta| \leq m$.  
  The definition of $p$ provides $p \in C^m(\overline{B_j(0)})$. The boundedness of $(p_{\e} )_{\e>0} \subseteq C^{m,s}(\Rn)$ and the pointwise convergence of $\p_x^{\alpha} p_{\e} \rightarrow \p_x^{\alpha}p$ if $\e \rightarrow 0$ for all $\alpha \in \N_0^n$ yields $p \in C^{m,s}(\Rn)$.
\end{proof}

The previous result enables us to show the next claim:

\begin{lemma}\label{lemma:convergence2}
   Let $m \in \N_0$ and $0<s\leq 1$. 
%    Additionally let $( \Omega_j \times A_i )_{i,j \in \N} $ be a countable open cover of bounded sets with Lipschitz boundary of $\RnRn$. Furthermore, let $( \p_x^{\beta} p_{\e} )_{\e>0} \subseteq C^{0,s}(\RnRnx{x}{\xi})$ be a bounded sequence for all $\beta \in \Non$ with $|\beta| \leq m$. 
   Then there is a subsequence $(p_{\e_k} )_{k \in \N} \subseteq (p_{\e} )_{\e>0}$ with $\e_k \rightarrow 0$ for $k \rightarrow \infty$ and a $p \in C^{0,s}(\RnRn)$ such that for all $\beta \in \N_0^n$ with $|\beta| \leq m$ we have
  \begin{enumerate}
    \item[i)] $\p_x^{\beta} p \in C^{0,s}(\RnRnx{x}{\xi})$, 
    \item[ii)] $\p^{\beta}_x  p_{\e_k} \xrightarrow[]{k \rightarrow \infty} \p^{\beta}_x p$ converges uniformly on each compact set $K \subseteq \RnRn$. 
  \end{enumerate}
\end{lemma}

\begin{proof}
  It is sufficient to show the claim for all sets $\overline{B_j(0) \times B_i(0)  }$, $i,j \in \N$. Since the subset $( \p_x^{\beta } p_{\e} )_{\e>0} $ is bounded in $ C^{0,s}(\RnRnx{x}{\xi})$, we iteratively conclude from Lemma \ref{lemma:convergence} the existence of a subsequence $(p_{\e_k} )_{k \in \N}$ of $(p_{\e} )_{e>0}$ and of functions $q_{ \beta } \in C^{0,s}(\RnRnx{x}{\xi})$ such that
  \begin{align}\label{eq3}
    \p^{ \beta }_x  p_{\e_k} \xrightarrow[]{k \rightarrow \infty} q_{ \beta } \qquad \text{uniformly in } \overline{  B_j(0) \times B_i(0) } 
  \end{align}
  for all $i,j \in \N$  and $\beta \in \Non$ with $|\beta| \leq m$. Choosing an arbitrary but fixed $\xi \in \Rn$, (\ref{eq3}) implies the uniformly convergence of
  \begin{align}\label{eq4}
    \p^{ \beta }_x  p_{\e_k}(., \xi) \xrightarrow[]{k \rightarrow \infty} q_{ \beta }(.,\xi)
  \end{align}
   in $\overline{B_j(0)}$ for all $\beta \in \Non$ with $|\beta| \leq m$ and all $j \in \N$. Hence $( p_{\e_k}(., \xi) )_{k \in \N}$ is a Cauchy sequence in $C^m(\overline{B_j(0)})$. Due to the completeness of $C^m (\overline{B_j(0)})$ we have the convergence of $( p_{\e_k}(., \xi) )_{k \in \N}$ to $\tilde{p}$ in $C^m (\overline{B_j(0)})$. Consequently we obtain for all $\beta \in \Non$ with $|\beta| \leq m$ and each $j \in \N$:
  \begin{align}
    \p^{ \beta }_x  p_{\e_k}(., \xi) \xrightarrow[]{k \rightarrow \infty} \p^{ \beta }_x \tilde{p}
  \end{align}
  in $C^0(\overline{B_j(0)})$. Because of the uniqueness of the strong limit we get together with (\ref{eq4}) that $\p^{ \beta }_x \tilde{p} = q_{\beta}(.,\xi)$ for each $\beta \in \Non$ with $|\beta| \leq m$. Thus with $p(x,\xi):= q_0(x,\xi)$ for all $x,\xi \in \Rn$ the claim holds.
\end{proof}

Finally we are able to show the main theorem of this subsection: 

\begin{thm}\label{kor:konvergenz}\label{thm:pointwiseConvergence}
  Let $m \in \N_0$, $M \in \N \cup \{ \infty \}$ and $0<s\leq 1$. 
  Furthermore, let $(p_{\e} )_{\e>0} \subseteq C^{m,s}S^0_{0,0}(\RnRn ;M)$ be a bounded sequence. 
  Then there is a subsequence $(p_{\e_l} )_{l \in \N} \subseteq (p_{\e} )_{\e>0}$ with $\e_l \rightarrow 0$ for $l \rightarrow \infty$ and a function $p: \Rn_x \times  \Rn_{\xi} \rightarrow \C$ such that for all $\alpha, \beta \in \N_0^n$ with $|\beta| \leq m$ and $|\alpha| \leq M-1$ we get
  \begin{itemize}
    \item[i)] $\p_x^{\beta}\pa{\alpha} p$ exists and $\p_x^{\beta} \pa{\alpha} p \in C^{0,s}(\RnRn)$,
    \item[ii)] $\p^{\beta}_x \pa{\alpha} p_{\e_l} \xrightarrow[]{l \rightarrow \infty} \p^{\beta}_x \pa{\alpha} p$ is uniformly convergent on each compact set $K \subseteq \RnRn$.
  \end{itemize}
  In particular $p \in C^{m,s}S^0_{0,0}( \RnRn; M-1)$.
\end{thm}

\begin{proof}
  It is sufficient to prove the claim for $\overline{B_j(0) \times B_{j}(0) }$, $j \in \N$. Applying Lemma \ref{lemma:boundOfSymbolsOnCs} we get for all $\beta, \gamma \in \Non$ with $|\beta| \leq m$ and $|\gamma| \leq M-1$ the boundedness of the sequence $( \p_x^{\beta} \pa{\gamma} p_{\e} )_{\e>0} \subseteq C^{0,s}(\RnRn)$. Thus by Lemma \ref{lemma:convergence2} we inductively obtain the existence of a subsequence $(p_{\e_l} )_{l \in \N} \subseteq (p_{\e} )_{\e>0}$ and functions $q_{\alpha} \in C^{0,s}(\RnRn)$ with the following properties: For all $j,j \in \N$ and $\alpha,\beta \in \N_0^n$ with $|\beta| \leq m$ and $|\alpha| \leq M-1$ we have $\p_x^{\beta} q_{\alpha} \in C^{0,s}(\RnRnx{x}{\xi})$ and 
  \begin{align}\label{27}
    \p^{\beta}_x \pa{\alpha }  p_{\e_l} \xrightarrow[]{l \rightarrow \infty} \p^{\beta}_x q_{\alpha }
  \end{align}
  converges uniformly on $\overline{B_j(0) \times B_{j}(0)}$. 
  Now we choose an arbitrary but fixed $k \in \N_0$ with $k \leq M-1$ and $x \in \Rn$. The boundedness of $( \pa{\gamma}  p_{\e_l})_{l \in \N} \subseteq C^{0,s}(\RnRn)$ for all $\gamma \in \Non$ with $|\gamma| \leq k$ leads to 
  \begin{align*}
    \| p_{\e_l}(x,.) \|_{ C^{k,s}(\Rn) } 
    \leq \max_{\substack{ \gamma \in \N_0^n \\ |\gamma| \leq k} } \| \pa{\gamma} p_{\e_l} \|_{ C^{0,s}(\RnRn) }
    \leq C_{k}
  \end{align*} 
  for all $x\in \Rn$ and $l \in \N$. 
  By means of Lemma \ref{lemma:convergence} we obtain via a diagonal sequence argument the existence of a subsequence of $( p_{\e_l} )_{l \in \N}$ denoted by $( p_{\e_{l_r}} )_{r \in \N}$ and of a function $\tilde{p} \in C^{M-1}(\Rn)$ with the property
  \begin{align}\label{28}
    \pa{\gamma} p_{\e_{l_r}}(x,\xi) \xrightarrow[]{r \rightarrow \infty} \pa{\gamma} \tilde{p}(\xi) \qquad \textrm{pointwise for all } \xi \in \Rn 
  \end{align}
  and every $\gamma \in \Non$ with $|\gamma| \leq M-1$. On account of (\ref{27}) and (\ref{28}) the uniqueness of the limit gives us $q_{\alpha}(x,.)= \pa{\alpha } \tilde{p}$. This implies $p(x,.):= q_0(x,.) \in C^{M-1}(\Rn)$ for all $x \in \Rn$, $(i)$ and $(ii)$. Note that $(i)$ implies 
  for all $ \gamma \in \Non$ with $|\gamma| \leq M-1 $
  \begin{align*}
    \| \pa{\gamma}p(.,\xi) \|_{ C^{m,s}(\Rn) } \leq \max_{|\beta| \leq m}\| \p_x^{\beta} \pa{\gamma} p \|_{ C^{0,s}(\RnRn) } \leq C_{\gamma} \qquad \text{ for all } \xi \in \Rn.
  \end{align*}
  Consequently $p \in C^{m,s}S^0_{0,0}( \RnRn; M-1)$.
\end{proof}

\input{SymbolReduktion}

\input{PropertiesOfTepsilon}

\subsection[Characterization of Operators with Symbols in $C^s S^m_{0,0}$]{Characterization of Pseudodifferential Operators with Symbols in $C^s S^m_{0,0}$}\label{section:classificationA00}

In this subsection we will first prove the 
the characterization of pseudodifferential operators with symbols of the symbol-class $C^s \Snn{0}{0}{0}{M}$. Then the result is extended to non-smooth pseudodifferential operators of the class $C^s \Snn{m}{0}{0}{M}$ of the order $m$. In the non-smooth case, one is confronted with the following problem: In general we do not have the continuity of non-smooth pseudodifferential operators with coefficients in a Hölder space as a map from $H^m_q(\Rn)$ to $L^q(\Rn)$. But every element of the set $ \mathcal{A}^{m,M}_{0,0}(\tilde{m},q)$ is a linear and bounded map from $H^m_q(\Rn)$ to $L^q(\Rn)$. Hence this ansatz just provides a characterization of those non-smooth pseudodifferential operators which are linear and bounded as maps from $H^m_q(\Rn)$ to $L^q(\Rn)$. As already mentioned, the proof relies on the main 
idea of 
the proof in the smooth case by Ueberberg \cite{Ueberberg}.\\

\begin{thm} \label{thm:classA00}
  Let $1<q<\infty$ and $m \in \N_0$ with $m>n/q$. Additionally let $M \in \N \cup \{ \infty \}$ with $M  >n$.  We define $\tilde{M}:= M-(n+1)$. Considering $T \in \mathcal{A}^{0,M}_{0,0}(m,q)$ and $\tilde{M}\geq 1$, we get for all $0< \tau \leq m- n/q$ with $\tau \notin \N_0$
  \begin{align*}
    T\in \op C^{\tau}S^0_{0,0}(\RnRn; \tilde{M}-1) \cap \mathscr{L}(L^q(\Rn)).
  \end{align*}
  \vspace*{-1cm}

\end{thm}

\begin{proof}
  Let $\tau \in (0,m-n/q]$ with $\tau \notin \N$ be arbitrary but fixed. 
  Let $T_{\e}$, $\e \in (0,1]$, be as in Subsection \ref{section:PropertiesOfTe}.
  Then  $T_{\e} : \sd \rightarrow \s$ is continuous, cf. \cite[Lemma 5.28]{Diss}. 
  The proof of this theorem is divided into three different parts. First we write $T_{\e}$ as a pseudodifferential operator with a double symbol. In step two we reduce the double symbol to an ordinary symbol $p_{\e}$ of $T_{\e}$. Finally, we conclude the proof in part three. Here we use the pointwise convergence of a subsequence of $( p_{\e} )_{\e > 0}$ to get a symbol $p$ with the property $p(x,D_x)u = Tu$ for all $u \in \s$. \\

  We begin with step one: Since $T_{\e} : \sd \rightarrow \s$ is continuous, Theorem \ref{thm:SchwartzKernel} gives us the existence of a Schwartz-kernel $t_{\e} \in \snn$ of $T_{\e}$. Thus 
  \begin{align}\label{36}
    T_{\e} u(x) = \int t_{\e}(x,y) u(y) dy \qquad \text{for all } u \in \s \text{ and all } x\in \Rn.
  \end{align} 
  Now we choose $u,g \in \s$ with $g(0)=1$ and $g(-x) = g(x)$ for all $x\in \R$. We define $g_y: \Rn \rightarrow \C$ for $y \in \Rn$ by $g_y:= \tau_y(g)$.
  Next let $x \in \Rn$ be arbitrary, but fixed. Then we define 
  \begin{align*}
    h(z) := u(z) g_z(x) \qquad \text{for all } z \in \Rn. 
  \end{align*}
  Using the inversion formula, cf. e.g. \cite[Example 3.11]{PDO},
  we obtain
  \begin{align*}
    u(x)=h(x)= \osint e^{i(x-y) \cdot \xi} h(y) dy \dq \xi = \osint e^{i(x-y) \cdot \xi} u(y) g_y(x) dy \dq \xi.  
  \end{align*}
  If we first insert the previous equality in  (\ref{36}) and use the definition of the oscillatory integrals, integration by parts with respect to $y$ and Lebesgues theorem afterwards, we get
  \begin{align*}  
   T_{\e} u(x) 
    &= \int t_{\e}(x,z) \left[  \osint e^{i(z-y) \cdot \xi} u(y) g_y(z) dy \dq \xi \right] dz \\ 
    &= \lim_{\alpha \rightarrow 0} \int t_{\e}(x,z) \cdot \iint e^{-iy\cdot \xi}  e^{iz\cdot \xi} u(y) g_y(z) \chi(\alpha y, \alpha \xi) dy \dq \xi \, dz \\ 
     &= \lim_{\alpha \rightarrow 0}  \iint e^{-iy\cdot \xi} \chi(\alpha y, \alpha \xi) \left[T_{\e}( e_{\xi}  g_y)(x)\right]  u(y) dy \dq \xi,
  \end{align*}
    where $\chi \in \snn$ with $\chi(0,0)=1$.   
  Defining $p_{\e,0}: \RnRnRn \rightarrow \C$ by 
  $$p_{\e,0}(x,\xi,y) := e^{-ix\cdot \xi} T_{\e}( e_{\xi}  g_y)(x) \qquad \text{ for all } x, \xi, y \in \Rn,$$ 
  we conclude 
  \begin{align*}
    T_{\e} u(x) 
    = \osint e^{i(x-y)\cdot \xi} p_{\e,0}(x,\xi,y)  u(y) dy \dq \xi.
  \end{align*}
  Here $p_{\e,0}$ is the double symbol of $T_{\e}$, cf.\;Lemma \ref{lemma:PseudosMitDoppelsymbol}, as we will see in step two.
  \\

  Secondly we want to construct for all $0< \e \leq 1$ symbols $p_{\e} \in C^{\tau} \Snn{0}{0}{0}{\tilde{M}}$, with 
  \begin{itemize}
    \item[i)] $T_{\e} u = p_{\e}(x,D_x) u$ for all $u \in \s$,
    \item[ii)] $( p_{\e})_{ 0< \e \leq 1 }$ is a bounded sequence of $ C^{\tau} \Snn{0}{0}{0}{\tilde{M}}$.
  \end{itemize}
  Since $T_{\e}: \sd \rightarrow \s$ is linear and continuous
  and because of Proposition \ref{prop:SmoothnessOfPepsilon0},
  we can apply Lemma \ref{lemma:AbleitungVonp} and Lemma \ref{lemma:ContinuityOfIteratedCommutatorOfTeFromLqToLq}
  and get for $\alpha, \gamma \in \Non$ with $|\alpha| \leq M$:
  \begin{align*}
    &\left\|  \pa{\alpha} D_y^{\gamma} p_{\e,0}(.,\xi,y) \right\|^q_{ C^{\tau} } \leq \left\| \pa{\alpha} D_y^{\gamma} p_{\e,0}(.,\xi,y) \right\|^q_{ H^m_q } 
      \leq \sum_{ |\beta| \leq m } \left\| \pa{\alpha} D_x^{\beta} D_y^{\gamma} p_{\e,0}(x,\xi,y) \right\|^q_{ L^q(\Rn_x) } \\
    & \qquad \leq \sum_{ |\beta| \leq m } \sum_{ \beta_1 + \beta_2 =\beta} \left\| C_{\beta_1,\beta_2} \left[ \ad(-ix)^{\alpha} \ad(D_x)^{\beta_1} T_{\e} \right]\left(e^{ix \cdot \xi} D_x^{\beta_2 + \gamma} g_y \right) (x) \right\|^q_{ L^q(\Rn_x) }
    \leq C_{\alpha,m,\gamma}
  \end{align*}
  for all $\xi, y \in \Rn$ and $0< \e \leq 1$. 
  Hence $\{ p_{\e,0} : 0< \e \leq 1\} $ is a bounded subset of $C^{\tau} S^{0}_{0,0}(\RnRn \times \Rn;M)$. Now we define 
  \begin{align*}
    p_{\e} (x,\xi) := \osint e^{-iy \cdot \eta} p_{\e,0}(x,\xi + \eta,x+y) dy \dq \eta \qquad \text{for all } x,\xi \in \Rn.
  \end{align*}
  An application of Theorem \ref{thm:SymbolReduktionNichtGlatt} and Theorem \ref{thm:aLInCsSm00} yields the properties i) and ii).  
  So we can turn to step three now.\\ 

  On account of ii) it is possible to apply Lemma \ref{thm:pointwiseConvergence} which yields the existence of a subsequence $( p_{\e_k} )_{k \in \N}$ of $( p_{\e})_{0 < \e \leq 1}$ with $\e_k \rightarrow 0$ if $k \rightarrow \infty$ such that 
  \begin{align}\label{39}
    p_{\e_k} \xrightarrow[]{k \rightarrow \infty} p \qquad \text{pointwise},
  \end{align}
  where $p \in C^{\tau} \Snn{0}{0}{0}{ \tilde{M}-1 }$. Let $u \in \s$  be arbitrary.  
  Because of  (\ref{39})  and the boundedness of $( p_{\e_k} )_{k \in \N} \subseteq  C^{\tau} \Snn{0}{0}{0}{\tilde{M}}$, we get
  \begin{align}\label{41}
    p_{\e_k}(x,D_x) &u 
    \xrightarrow[]{k \rightarrow \infty} 
      p(x,D_x) u
  \end{align} 
  pointwise due to Lebesgue's theorem.
  Choosing $N \in \N$ with $n< 2N \leq M$
  we get by Lemma \ref{lemma:AbschatzungDoppelSymbol}: 
  \begin{align*}
    |p_{\e_k}(x,D_x)u(x) - p(x,D_x)u(x)|^q &\leq \left( |p_{\e_k}(x,D_x)u(x)| + \lim_{k \rightarrow \infty}|p_{\e_k}(x,D_x)u(x)| \right)^q \\
    &\leq C_{N,n} \<{x}^{-2Nq} \in L^1(\Rn_x)
  \end{align*}
  for all $k \in \N$. 
  Together with (\ref{41}) 
  we can apply of Lebesgue's theorem and obtain
  \begin{align*}
    &\| p_{\e_k}(x,D_x)u - p(x,D_x)u \|^q_{L^q(\Rn)} %\\
    = \intr | p_{\e_k}(x,D_x)u(x) - p(x,D_x)u(x) |^q dx  \xrightarrow[]{k \rightarrow \infty}  0
  \end{align*}
  Together with i) and Lemma \ref{lemma:LqLimitOfTepsilonu} we conclude
  \begin{align*}
    p(x,D_x)u = L^q-\lim_{k \rightarrow \infty} p_{\e_k}(x,D_x)u = L^q-\lim_{k \rightarrow \infty} T_{\e_k} u = Tu. 
  \end{align*}
\end{proof}

By means of order reducing operators we can extend the previous characterization to 
the class $C^s S^m_{0,0}$ for general $m$:

\begin{lemma} \label{lemma:classA00}
  Let $m\in \R$, $1<q<\infty$, $\tilde{m} \in \N_0$ with $\tilde{m}>n/q$. Additionally let $M \in \N_0 \cup \{ \infty \}$ with $M >n$. We define $\tilde{M}:=M-(n+1)$. Considering  $T \in \mathcal{A}^{m,M}_{0,0}(\tilde{m},q)$ and $\tilde{M} \geq 1$ we have for $s \in (0, \tilde{m}-n/q]$ with $s\notin \N_0$:
  $$T \in \op C^s S^m_{0,0}(\RnRn; \tilde{M}-1) \cap \mathscr{L}(H^m_q(\Rn),L^q(\Rn)).$$
\end{lemma}

\begin{proof}
  Let $ s \in ( 0, \tilde{m}-n/q ] $ with $s\notin \N_0$ and $\delta \in \N_0^n$. Due to  Remark \ref{bem:SymbolOfIteratedCommutator} 
  and Theorem \ref{thm:stetigInBesselPotRaum} we get that
  \begin{eqnarray} \label{13}
    \ad(-ix)^{\delta}\Lambda^{-m}: L^q(\Rn) \rightarrow H_q^{m+|\delta|}(\Rn) \subseteq H_q^{m}(\Rn) \textrm{ is continuous}.
  \end{eqnarray}
  Now let $l\in \N_0$, $\alpha_1,\ldots, \alpha_l \in \N_0^n$ and $\beta_1,\ldots, \beta_l \in \N_0^n$ such that $|\beta|\leq \tilde{m}$ and $|\alpha| \leq M$, where $\beta:= \beta_1 + \ldots + \beta_l$ and $\alpha:= \alpha_1 + \ldots + \alpha_l$. 
  Since $\ad(-ix)^{\tau_2}\ad(D_x)^{\delta} \Lambda^{-m} \equiv 0$ for every $\tau_2, \delta \in \Non$ with $|\delta| \neq 0$ due to Remark \ref{bem:SymbolOfIteratedCommutator}, we can iteratively show
  \begin{eqnarray*}
    \lefteqn{ \adc{\alpha}{\beta} (T\Lambda^{-m}) }\\
    &=& \hspace{-0,3cm} \sum_{
	  \substack{\gamma_1 + \delta_1 = \alpha_1 \vspace{-1.7mm} \\  \vdots \vspace{0.2mm} \\ \gamma_l + \delta_l = \alpha_l}
	 } 
	\hspace{-0,1cm} C_{\gamma_1,\ldots,\gamma_l} [\adc{\gamma}{\beta} T][\ad(-ix)^{\delta} \Lambda^{-m}],
  \end{eqnarray*}
  where $\delta$ is defined by $\delta:=\delta_1+ \ldots + \delta_l$. 
  Combining (\ref{13}) and $T \in \mathcal{A}^{m,M}_{0,0}(\tilde{m},q)$ we obtain the continuity of
  \begin{eqnarray*}
         \adc{\alpha}{\beta} (T\Lambda^{-m}) : L^q(\Rn) \rightarrow L^q(\Rn).
  \end{eqnarray*}
  Therefore $T\Lambda^{-m} \in \mathcal{A}^{0,M}_{0,0}(\tilde{m},q)$. If we use Theorem \ref{thm:classA00}, we get
  \begin{align*}
    T\Lambda^{-m} \in \op C^{s} S^0_{0,0}(\RnRn; \tilde{M}-1) \cap \mathscr{L}(L^q(\Rn)).
  \end{align*}
  On account of $\Lambda^{m} \in \op S^m_{1,0}(\RnRn)$ and Theorem \ref{thm:stetigInBesselPotRaum} we have 
  \begin{align*}
     T \in \op C^{s} S^m_{0,0}(\RnRn; \tilde{M}-1) \cap \mathscr{L}(H^m_q(\Rn),L^q(\Rn)).
  \end{align*}
  \vspace*{-1cm}

\end{proof}

\subsection[Characterization of Operators with Symbols in $C^s S^m_{1,0}$]{Characterization of Pseudodifferential Operators with Symbols in $C^s S^m_{1,0}$}\label{section:classificationA10}

In applications to partial nonlinear differential equations the pseudodifferential operators are often in the class $C^s S^m_{1,0}(\RnRn)$. As we have seen in Example \ref{bsp:ElementDerCharakterisierungsmenge}, these operators are elements of the set $\mathcal{A}^{m}_{1,0}(\lfloor s \rfloor,q)$ with $1 < q < \infty$. In the present subsection we show that operators belonging to the set $ \mathcal{A}^{m, M}_{1,0}(\tilde{m},q)$ for sufficiently large $\tilde{m}$ are non-smooth pseudodifferential operators of the order $m$ whose coefficients are in a Hölder space. As an ingredient we use $ \mathcal{A}^{m, M}_{1,0}(\tilde{m},q) \subseteq \mathcal{A}^{m, M}_{0,0}(\tilde{m},q)$. Consequently we may apply the characterization of the pseudodifferential operators of the class $C^s S^m_{0,0}(\RnRn, M)$ in order to obtain the following main result of this section:

\begin{thm}\label{thm:classificationA10}
  Let  $m\in \R$, $1<q<\infty$ and $\tilde{m} \in \N_0$ with $\tilde{m}>n/q$. Additionally let $M \in \N_0$ with $M>n$. We define $\tilde{M}:=M-(n+1)$. Assuming $P \in \mathcal{A}^{m, M}_{1,0}(\tilde{m},q)$ and $\tilde{M} \geq 1$, we obtain for all $ \tau \in \left(0,\tilde{m}-n/q \right]$ with $\tau \notin \N_0$: 
  $$P \in \op C^{\tau} S^m_{1,0}(\RnRn; \tilde{M}-1) \cap \mathscr{L}(H^m_q(\Rn),L^q(\Rn)).$$
\end{thm}

\begin{proof}
  Let $\tilde{m}- n/q \geq \tau > 0$ with $\tau \notin \N_0$ and $P \in \mathcal{A}^{m, M}_{1,0}(\tilde{m},q)$. Due to Lemma \ref{lemma:setA} we have
    $P \in \mathcal{A}^{m, M}_{1,0}(\tilde{m},q) \subseteq \mathcal{A}^{m, M}_{0,0}(\tilde{m},q).$
  Hence we get by means of Lemma \ref{lemma:classA00}:
  \begin{align*}
    P \in \op C^{\tau} S^m_{0,0}(\RnRn; \tilde{M}-1) \cap \mathscr{L}(H^m_q(\Rn),L^q(\Rn)).
  \end{align*} 
  Let $\alpha \in \N_0^n$ with $|\alpha| \leq \tilde{M}-1$. 
  Then $\ad(-ix)^{\alpha} P \in \mathcal{A}^{m-|\alpha|, M-|\alpha|}_{1,0}(\tilde{m},q) $. Because of Lemma \ref{lemma:setA} and Lemma \ref{lemma:classA00},
  we obtain
   \begin{align*}
     \ad(-ix)^{\alpha} P \in \op C^{\tau} S^{m-|\alpha|}_{0,0}(\RnRn; \tilde{M}-|\alpha|-1).
  \end{align*}
  Due to Remark \ref{bem:SymbolOfIteratedCommutatorNonSmooth} the symbol of $\ad(-ix)^{\alpha} P$ is $\pa{\alpha} p(x,\xi)$ if $p$ is the symbol of $P$. Hence
  \begin{align*}
    \| \pa{\alpha} p(.,\xi) \|_{  C^{\tau} (\Rn) } \leq C_{\alpha} \<{\xi}^{m-|\alpha|} \qquad \textrm{for all } \xi \in \R^n.
  \end{align*}
  Consequently $p$ is an element of $C^{\tau} S^{m}_{1,0}(\RnRn; \tilde{M}-1)$.
\end{proof}

In the case $\tilde{M}-1>\max\{n/2, n/q\}$, $1<q< \infty$, every pseudodifferential operator whose symbol is in the class $C^{\tau} S^m_{1,0}(\RnRn; \tilde{M}-1)$, where $\tau>0$ and $m \in \R$, is an element of $\mathscr{L}(H^m_q(\Rn),L^q(\Rn))$ due to Theorem \ref{thm:stetigInHoelderRaum}. Therefore we have in this case
\begin{align*}
  \op C^{\tau} S^m_{1,0}(\RnRn; \tilde{M}-1) \cap \mathscr{L}(H^m_q(\Rn),L^q(\Rn)) = \op C^{\tau} S^m_{1,0}(\RnRn; \tilde{M}-1).
\end{align*}

%% file: SymbolReduktion.tex
\subsection[Reduction of Non-Smooth Pseudodifferential Operators]{Reduction of Non-Smooth Pseudodifferential Operators with Double Symbol} \label{section:symbolReduction}

In this subsection we derive a formula representing an operator with a non-smooth double symbol as an operator with a non-smooth single symbol.  
During the development of this work (however independent) Köppl generalized this result in his diploma thesis, cf.\;\cite[Theorem 3.33]{Koeppl}, for non-smooth double symbols of the symbol-class $C^{\tilde{m}, \tau} S^{m,m'}_{\rho, \delta}(\RnRnRnRn; N)$ where $N=\infty$. However, as we will show, the assumption $N= \infty$ (as assumed in the diploma thesis) may be weakend. Then the smoothness in $\xi$ is reduced by the order of $n$.

For the proof of the characterization of non-smooth pseudodifferential operators only the case $\rho=\delta=0$ is required. Thus we restrict the symbol reduction to this case. This significantly simplifies some proofs. The main idea of the symbol reduction is taken from that one of the smooth case, cf.\;e.g.\;\cite[Theorem 2.5]{KumanoGo}. Since the symbols are non-smooth in both variables, the proof has to be adopted to this modified condition. 

We begin with some auxiliary tools needed for the proof of the symbol reduction:

\begin{lemma}\label{lemma:SymbolReduction}
  Let $s > 0$, $s \notin \N_0$, $m,m' \in \R$ and $N \in \N_0 \cup \{ \infty \}$. Additionally we
  choose $l,l_0, l_0' \in \N_0$ such that
  \begin{align*}
    -2 l +m < -n, \quad  - 2 l_0 < -n \quad \textrm{and} \quad -2l_0' + 2l + m' < -n. 
  \end{align*}
  Furthermore, let $P:= p(x,D_x,x',D_{x'}) \in \op C^s_*S^{m,m'}_{0,0}(\RnRn \times \RnRn; N)$. 
  For $u \in \s$ we define $\tilde{p}: \R^{5n} \rightarrow \C$ by
  \begin{align*}
    \tilde{p}(x,\xi,x', \xi',x''):=  \<{-\xi + \xi'}^{-2l} \<{D_{x'}}^{2l} \hat{p}(x,\xi,x', \xi', x'')
  \end{align*}
  for all $x,\xi,x',\xi', x'' \in \Rn$, where
  \begin{align*}
    \hat{p}(x,\xi,x', \xi', x''):= \<{x'-x''}^{-2l_0} \<{D_{\xi'}} ^{2l_0} \left[ \<{\xi'}^{-2l_0'} \<{D_{x''}}^{2l_0'} p(x,\xi,x', \xi') u( x'') \right]
  \end{align*}
  for all $x,\xi,x',\xi', x'' \in \Rn$. Then we have for all $x \in \Rn$:
  \begin{align*}
    Pu(x) = \iiiint e^{-i(x'-x) \cdot \xi -i (x''-x') \cdot \xi'} \tilde{p}(x,\xi,x', \xi',x'') dx'' \dq \xi' dx' \dq \xi.
  \end{align*}
\end{lemma}

\begin{proof}
  Let $x \in \Rn$ be arbitrary but fixed. 
    Additionally let $u \in \s$ and $\chi \in \mathcal{S}(\R^{4n})$ with $\chi(0) = 1$. For each $0 < \e < 1$ we denote $\chi_{\e}: \R^{4n} \rightarrow \C$ by
    \begin{align*}
      \chi_{\e}(\xi,\xi',y,y'):= \chi(\e \xi,\e \xi',\e y,\e y') \qquad \text{ for all } \xi,\xi',y,y' \in \Rn.
    \end{align*}
  We define $p_{\e,u}: \RnRnRnRn \times \Rn \rightarrow \C$ for every $0 < \e < 1$ by
  \begin{align*}
    p_{\e,u} (\tilde{x}, \xi, x', \xi', x'') := \chi_{\e}(\xi,\xi', x'-x , x''-x') p(\tilde{x}, \xi, x', \xi') u(x'').
  \end{align*}
  for all $\tilde{x}, \xi,x',\xi',x'' \in \Rn$. 
  Using Leibniz's rule, $p \in C^s_*S^{m,m'}_{0,0}(\RnRn \times \RnRn; N)$ and  $\chi \in \mathcal{S}(\R^{4n})$ provides for all $\alpha, \beta, \gamma \in \Non$: 
  \begin{align}\label{NR}
    \p^{\alpha}_{x''} \p^{\beta}_{\xi'} \p^{\gamma}_{x'}  p_{\e,u}(x,\xi,x',\xi', x'') \in L^1(\RnRnx{\xi}{x'} \times \RnRnx{\xi'}{x''}) .
  \end{align}  
  Due to the definition of the oscillatory integral, the change of variables $x':=x+y$ and $x'':= x' + y'$ and Fubini's theorem we obtain
  \begin{align}\label{p17}
    &Pu(x) = \osiint e^{-i(y \cdot \xi + y' \cdot \xi')} p(x,\xi,x+y,\xi') u(x+y+y')dy dy' \dq \xi \dq \xi' \notag \\
    &\quad = \lim_{\e \rightarrow 0} \iiiint e^{-i(y \cdot \xi + y' \cdot \xi')} \chi_{\e}(\xi,\xi',y,y') p(x,\xi,x+y,\xi') u(x+y+y')dy dy' \dq \xi \dq \xi' \notag \\
    &\quad = \lim_{\e \rightarrow 0} \iiiint e^{-i(x'-x) \cdot \xi -i (x''-x') \cdot \xi'} p_{\e,u}(x,\xi,x',\xi', x'') dx'' \dq \xi' dx' \dq \xi.
  \end{align} 
  Now we choose $l,l_0, l_0' \in \N_0$ as in the assumptions. Then we define for each $0< \e < 1$ the function $\tilde{p}_{\e}: \R^{5n} \rightarrow \C$ by
  \begin{align*}
    \tilde{p}_{\e}(\tilde{x},\xi,x', \xi',x''):=  \<{-\xi + \xi'}^{-2l} \<{D_{x'}}^{2l} \hat{p}_{\e}(\tilde{x},\xi,x', \xi', x'')
  \end{align*}
  for all $\tilde{x},\xi,x', \xi',x'' \in \Rn$, where the function $\hat{p}_{\e}: \R^{5n} \rightarrow \C$ is defined by
  \begin{align*}
    \hat{p}_{\e}(\tilde{x},\xi,x', \xi', x''):= \<{x'-x''}^{-2l_0} \<{D_{\xi'}} ^{2l_0} \left[ \<{\xi'}^{-2l_0'} \<{D_{x''}}^{2l_0'} p_{\e,u}(\tilde{x},\xi,x', \xi',x'') \right]
  \end{align*}
  for each $\tilde{x},\xi,x', \xi',x'' \in \Rn$. Additionally  we can integrate by parts in (\ref{p17}) due to (\ref{NR}) and get
  \begin{align}\label{p18}
    Pu(x) 
    = \lim_{\e \rightarrow 0} \iiiint e^{-i(x'-x) \cdot \xi -i (x''-x') \cdot \xi'} \tilde{p}_{\e}(x,\xi,x', \xi',x'') dx'' \dq \xi' dx' \dq \xi.
  \end{align}
  An application of the Leibniz rule, $p \in C^s_*S^{m,m'}_{0,0}(\RnRnRnRn; N)$ and $u \in \s$ 
  yields 
  the existence of a constant $C$, which is independent of $0< \e < 1$ 
  , such that
  \begin{align}\label{p20}
    &|e^{-i(x'-x) \cdot \xi -i (x''-x') \cdot \xi'} \tilde{p}_{\e}(x,\xi,x', \xi',x'')| 
     \leq C \<{\xi}^{m-2l} \<{\xi'}^{-2l_0' +m'+2l} \<{x'}^{-2l_0} \<{x''}^{2l_0 -M} \notag \\
    & \quad \in L^1(\RnRnx{x''}{\xi'} \times \RnRnx{x'}{\xi}).
  \end{align}
  Moreover,
  if we use the Leibniz rule, the pointwise convergence of $\chi_{\e}$ to $1$ and the pointwise convergence of every derivative of $\chi_{\e}$ to $0$, see e.g. \cite[Lemma 6.3]{KumanoGo}, we obtain 
  \begin{align}\label{p23}
    \tilde{p}_{\e}(x,\xi,x', \xi',x'') \xrightarrow[]{\e \rightarrow 0}  
    \tilde{p} (x,\xi,x', \xi',x'').
  \end{align}
  Hence applying Lebesgue's theorem to (\ref{p18}) concludes the proof.

\end{proof}

Making use of this integral representation we are able to show the following result:

\begin{lemma}\label{lemma:AbschatzungDoppelSymbol}
   Let $s > 0$, $s \notin \N_0$ and $m,m' \in \R$. Additionally let $N \in \N_0 \cup \{ \infty \}$ and $l' \in \N_0$ with $ l' \leq N$.  
  Furthermore, let $ \mathscr{B} \subset C^s_*S^{m,m'}_{0,0}(\RnRn \times \RnRn; N)$ be bounded and $u \in \s$. Assuming $p \in C^s_*S^{m,m'}_{0,0}(\RnRn \times \RnRn; N)$, we denote $P:= p(x,D_x, x', D_{x'})$.  
  Then we obtain the existence of a constant $C$, independent of $x \in \Rn$ and $p \in \mathscr{B}$, such that
  \begin{align*}
    |P u(x)| \leq C \<{x}^{-l'} \qquad \text{for all } x \in \Rn.
  \end{align*}
\end{lemma}

\begin{proof}
  An application of Lemma \ref{lemma:SymbolReduction}, Remark \ref{bem:eFunktion} and Remark \ref{bem:Absch} and of integration by parts with respect to $\xi$
  concludes the claim similarly to \cite[Lemma 5.14]{Diss}. 
\end{proof}

As in the smooth case, cf.\;e.g.\;\cite{KumanoGo}, we define for all $a \in C^{\tilde{m},s}_*S^{m}_{0,0}(\RnRnRn; N)$ the function $a_L(x,\xi):= \osint e^{-iy \cdot \eta} a(x, \eta + \xi, x+y) dy \dq \eta$. In order to verify that $a_L$ is a non-smooth single symbol, we need the next results:

\begin{prop}\label{prop:HilfslemmaIntAbschatzung}
  Let $m \in \R$ and $X$ be a Banach space with $X \hookrightarrow L^{\infty}(\Rn)$.
  Let $ l_0 \in \N_0$ with $-l_0 < -n$ and $\mathscr{B}$ be a set of functions $r: \RnRnRnRn \rightarrow \C$ which are smooth with respect to the fourth variable such that the next inequality holds for all $l \in \N_0$:
  \begin{align}\label{Stern}
    \| \<{D_y}^{2l} r(.,\xi,\eta,y) \|_X \leq C_{l} \<{y}^{-l_0} \<{\xi + \eta}^m \qquad \text{for all } \xi, \eta, y \in \Rn, r \in \mathscr{B}.
  \end{align}
  Then $ \int e^{-iy\cdot \eta } r(x,\xi, \eta,y) dy \in L^1(\Rn_{\eta})$ for all $x,\xi \in \Rn$. If we define  
  \begin{align*}
    I(x,\xi) := \int \left[ \int e^{-iy\cdot \eta } r(x,\xi,\eta,y) dy \right] \dq \eta 
  \end{align*}
  for arbitrary $x, \xi \in \Rn$ and $r \in \mathscr{B}$ we have
  \begin{align*}
    \left\| I(.,\xi) \right\|_{X} \leq C \<{\xi}^m \qquad \text{ for all } \xi \in \Rn \text{ and } r \in \mathscr{B}.
  \end{align*}
\end{prop}

\begin{proof}
  Let $\xi \in \Rn$ and $r \in \mathscr{B}$. Making use of the assumptions of the proposition we can show $\<{D_y}^{2\tilde{l}} r(x,\xi,\eta,y) \in L^1(\Rn_y)$ for each $x, \xi, \eta \in \Rn$ and $\tilde{l} \in \N_0$ due to (\ref{Stern}). 
  Consequently we can integrate by parts and we obtain for all $l \in \N_0$ and $x,\eta \in \Rn$:
  \begin{align}\label{p27}
    \int e^{-iy \cdot \eta} r(x,\xi,\eta,y) dy = \<{\eta}^{-2l} \int e^{-iy \cdot \eta} \<{D_y}^{2l} r(x,\xi,\eta,y) dy.
  \end{align} 
  Now we choose an $l \in \N_0$ with $|m| - 2l < -n$. Then we conclude
  \begin{align*}
    \|I(.,\xi)\|_{X} &= \left\| \int \<{\eta}^{-2l} \int e^{-iy \cdot \eta} \<{D_y}^{2l} r(.,\xi,\eta,y) dy \dq \eta \right\|_{X} \\
    &\leq \int \<{\eta}^{-2l} \int \left\| \<{D_y}^{2l} r(.,\xi,\eta,y) \right\|_{X} dy \dq \eta \\
    &\leq C \<{\xi}^{m} \int \<{\eta}^{-2l + |m|}  \int \<{y}^{-l_0} dy \dq \eta
    \leq  C \<{\xi}^m \qquad \text{for all } x, \xi \in \Rn, r \in \mathscr{B}.
  \end{align*}
  In particular this provides $\int e^{-iy \cdot \eta}r(x,\xi, \eta,y) dy \in L^1(\Rn_{\eta})$ for all $x,\xi \in \Rn$.
\end{proof}

\begin{bem}\label{bem:IntAbschatzung}
  In particular we can apply Proposition \ref{prop:HilfslemmaIntAbschatzung} on $X:= C^0_b(\Rn)$ and on the function $r: \RnRnRnRn \rightarrow \C$  defined by
  \begin{align*}
    r(x,\xi,\eta,y):= A^{l_0}(D_{\eta}, y) a(x,\xi+\eta,x+y) \qquad \text{for all } x, \xi, \eta, y \in \Rn.
  \end{align*}
\end{bem}

\begin{prop}\label{prop:OsziInt=Int}
  Let $0< s < 1$, $\tilde{m} \in \N_0$ and $m \in \R$. Additionally let $N \in \N_0 \cup \{ \infty \}$ with $n <  N$. Moreover, let $a \in C^{\tilde{m},s}_*S^{m}_{0,0}(\RnRnRn; N)$. Considering an $l_0 \in \N_0$ with $n< l_0 \leq N$, we define $r: \RnRnRnRn \rightarrow \C$ as in Remark \ref{bem:IntAbschatzung}.
  Then $ \int e^{-iy\cdot \eta } r(x,\xi, \eta,y) dy \in L^1(\Rn_{\eta})$ for all $x,\xi \in \Rn$ and we obtain 
  \begin{align*}
    \osint e^{-iy\cdot \eta } r(x,\xi,\eta,y) dy \dq \eta = \int \left[ \int e^{-iy\cdot \eta } r(x,\xi,\eta,y) dy \right] \dq \eta.
  \end{align*}
\end{prop}

\begin{proof}
  On account of $a(x, \eta + \xi, x+y) \in \mathscr{A}^{m,N}_0(\RnRnx{y}{\eta}) $ for all $x,\xi \in \Rn$ Theorem \ref{thm:OscillatoryIntegralGleichung} yields the existence of the oscillatory integral. Assuming an arbitrary $\chi \in \s$ with $\chi(0)=1$, we get for fixed $x,\eta, \xi \in \Rn$:
  \begin{align}\label{p61}
    e^{-iy\cdot \eta } \chi(\e y) r(x,\xi,\eta,y) \xrightarrow[]{\e \rightarrow 0} e^{-iy\cdot \eta } r(x,\xi,\eta,y) \quad \text{ for all } y \in \Rn.
  \end{align}
  Now let $0 < \e \leq 1$. We can prove the next two estimates if we use $\chi \in C^{\infty}_b(\Rn)$ and Remark \ref{bem:Absch}:
  \begin{align}
    | \p_y^{\alpha} r(x,\xi,\eta,y)| &\leq C_{\alpha,m} \<{y}^{-l_0} \<{\xi}^m \<{\eta}^{|m|} \qquad \text{for all } \alpha \in \Non,\label{p62}\\
    | \<{D_y}^{2l'} [\chi(\e y) r(x,\xi,\eta,y)]| &\leq C_{l',m} \<{y}^{-l_0} \<{\xi}^m \<{\eta}^{|m|} \qquad \text{for all } l' \in \N_0, \label{p63}
  \end{align}
  uniformly in $x,\xi, \eta, y \in \Rn$ and in $0 < \e \leq 1$. 
  Using $\chi \in \s \subseteq C^{\infty}_b(\Rn) $ and 
  integration by parts, which is possible because of (\ref{p63}), 
  first and (\ref{p63}) 
  provides for fixed $x,\xi \in \Rn$ and an arbitrary $l \in \N_0$ with $|m|-2l < -n$:
  \begin{align}\label{p65}
    &\left| \chi(\e \eta) \hspace{-0.5mm} \int \hspace{-0.5mm} e^{-iy\cdot \eta } \chi(\e y) r(x,\xi,\eta,y) dy \right| \leq C \hspace{-1mm} \int \hspace{-0.5mm} \left|  e^{-iy\cdot \eta } \<{\eta}^{-2l} \<{D_y}^{2l} [\chi(\e y) r(x,\xi,\eta,y)] \right| \hspace{-0.5mm} dy \notag \\
    &\qquad \leq C_{l,m,\xi} \<{\eta}^{-2l + |m|} \int \<{y}^{-l_0} dy \leq C_{l,m,\xi} \<{\eta}^{-2l + |m|} \in L^1(\Rn_{\eta}).
  \end{align}
  Here the constant $C_{l,m,\xi}$ is independent of $\e \in (0, 1]$. Setting $l'=0$, (\ref{p63})
  provides for each fixed $x,\xi,\eta \in \Rn$, that 
  $\{ y \mapsto \chi(\e y) r(x,\xi,\eta,y): 0 < \e \leq 1 \}$ 
  has a $L^1(\Rn_y)$-majorant. 
  Applying  Lebesgue's theorem we obtain
  \begin{align}\label{p66}
    \chi(\e \eta) \int e^{-iy\cdot \eta } \chi(\e y) r(x,\xi,\eta,y) dy \xrightarrow[]{\e \rightarrow 0} \int e^{-iy\cdot \eta } r(x,\xi,\eta,y) dy
  \end{align}
  for all $x,\xi,\eta \in \Rn$. Applying Lebesgue's theorem again we get for all $x,\xi \in \Rn$:
  \begin{align*}
    \osint e^{-iy\cdot \eta } r(x,\xi,\eta,y) dy \dq \eta &= \lim_{\e \rightarrow 0} \int \chi(\e \eta) \int e^{-iy\cdot \eta } \chi(\e y) r(x,\xi,\eta,y) dy \dq \eta \\
    &= \int \left[ \int e^{-iy\cdot \eta } r(x,\xi,\eta,y) dy  \right]\dq \eta.
  \end{align*}
  The assumptions of Lebesgue's theorem are fulfilled because of (\ref{p65}) and (\ref{p66}).
\end{proof}

The previous results enable us to show the following statement:

\begin{lemma}\label{lemma:AbschVonaL}
  Let $0< s < 1$, $\tilde{m} \in \N_0$ and $m \in \R$. Additionally let $\mathscr{B}$ be a bounded subset of $ C^{\tilde{m},s}_*S^{m}_{0,0}(\RnRnRn; N)$ and $N \in \N_0 \cup \{ \infty \}$  with $n < N $. We define for each $a \in \mathscr{B}$ the function $a_L: \RnRn \rightarrow \C$ by
  \begin{align*}
    a_L(x,\xi) := \osint e^{-iy \cdot \eta} a(x, \eta + \xi, x+y) dy \dq \eta \qquad \text{for all } x, \xi \in \Rn.
  \end{align*}
  Then there is a constant $C$, independent of $x,\xi \in \Rn$ and $a \in \mathscr{B}$, such that
  \begin{align*}
    |\p^{\delta}_x a_L (x, \xi)| \leq C \<{\xi}^m \qquad \text{for each } \delta \in \Non \text{ with } |\delta| \leq \tilde{m}.
  \end{align*}
\end{lemma}

Note that Theorem \ref{thm:ExistenceOfOscillatoryIntegral} yields the existence of $a_L(x, \xi)$ for all $x,\xi \in \Rn$ since $a(x, \eta + \xi, x+y) \in \mathscr{A}^{m, N}_0(\RnRnx{y}{\eta})$.
For the proof of Lemma \ref{lemma:AbschVonaL} we need:

\begin{lemma}\label{lemma:VertauscheAblNachXUndOsziInt}
  Let $N \in \N_0 \cup \{ \infty \}$ with $n < N$. Assuming $a \in C^{ \tilde{m} ,s} S^m_{\rho,0}(\RnRnRn;N)$ with $\tilde{m} \in \N_0$, $m \in \R$, $0 \leq \rho \leq 1$ and $0<s<1$, we define $a_L:\RnRn \rightarrow \C$ as in Lemma \ref{lemma:AbschVonaL}.
  Then we get for each $\beta \in \Non$ with $|\beta| \leq \tilde{m}$:
  \begin{align*}
    \p^{\beta}_x a_L(x,\xi) = \osint e^{-iy \cdot \eta} \p^{\beta}_x \{ a(x, \eta + \xi, x+y) \} dy \dq \eta \qquad \text{for all } x, \xi \in \Rn.
  \end{align*}
\end{lemma}

\begin{proof}
  The claim follows from Corollary \ref{kor:VertauschbarkeitVonOsziIntUndLimes} and from approximation of the function $\p_{x_j} \{ a(x,\xi + \eta, x+y) \}$ by difference quotients.
\end{proof}

 Now we are able to prove Lemma \ref{lemma:AbschVonaL}:

\begin{proof}[Proof of Lemma \ref{lemma:AbschVonaL}]
  Using Theorem \ref{thm:OscillatoryIntegralGleichung}, Proposition \ref{prop:OsziInt=Int} and Remark \ref{bem:IntAbschatzung} we get for each $l_0 \in \N_0$ with $n< l_0 \leq N$
  \begin{align*}
    |a_L(x,\xi)|
      &= \left| \osint e^{-iy \cdot \eta} A^{l_0}(D_{\eta}, y)  a(x,\xi+\eta,x+y) dy \dq \eta \right| \\
      &= \left| \iint e^{-iy \cdot \eta} A^{l_0}(D_{\eta}, y) a(x,\xi+\eta,x+y) dy \dq \eta \right|
    \leq C \<{\xi}^m
  \end{align*}
  for all $x,\xi \in \Rn$ and of $a \in \mathscr{B}$.
  Thus the claim holds for $\delta =0$. Now we assume $\delta \in \Non$ with $|\delta| \leq \tilde{m}$. 
  Due to Remark \ref{bem:AbleitungVonPDOWiederPDO} $\mathscr{B}^{\delta}:= \left\{ \p^{\delta}_x a : a \in \mathscr{B} \right\} $ is bounded in $ C^{\tilde{m} - |\delta|,s}_*S^{m}_{0,0}(\RnRnRn; N)$.  On account of Lemma \ref{lemma:VertauscheAblNachXUndOsziInt} the case $\delta=0$ applied on the set $\mathscr{B}^{\delta}$, gives us
  \begin{align*}
    |\p^{\delta}_x a_L (x, \xi)| \leq C \<{\xi}^m \qquad \text{for all } x, \xi \in \Rn \text{ and } a \in \mathscr{B}.
  \end{align*}
  \vspace*{-1cm}

\end{proof}

Having in mind the definition of the Hölder spaces, we need the next two statements to show that $a_L$ is a non-smooth symbol whose coefficient is in a Hölder space:

\begin{prop}\label{prop:fürAbschVonHoelderNorm}
  Let $0 < s < 1$, $\tilde{m} \in \N_0$, $N \in \N_0 \cup \{ \infty \}$ and $m \in \R$. Moreover, let $\mathscr{B} \subseteq C^{\tilde{m},s}S^{m}_{0,0}(\RnRnRn; N)$ be bounded. Then we have for each $\gamma, \beta \in \Non$ with $|\beta| \leq N$
  \begin{align*}
    \max_{|\alpha| \leq \tilde{m}} \left\{ \frac{ | \p_{x_1}^{\alpha} \p_y^{\gamma} \p_{\eta}^{\beta} a(x_1, \xi + \eta, x_1 + y) - \p_{x_2}^{\alpha} \p_y^{\gamma} \p_{\eta}^{\beta} a(x_2, \xi + \eta, x_2 + y) | }{|x_1 - x_2|^s}  \right\} \leq C \<{\xi + \eta}^m
  \end{align*}
  for all $x_1,x_2, y, \xi, \eta \in \Rn$ with $x_1 \neq x_2$ and $a \in \mathscr{B}$.
\end{prop}

\begin{proof}
  First of all we choose arbitrary $\alpha, \beta, \gamma \in \Non$ with $|\alpha| \leq \tilde{m}$ and $|\beta| \leq N$ and let $x_1, x_2 \in \Rn$. The boundedness of $\mathscr{B} $ in the set $C^{\tilde{m},s} S^{m}_{0,0}(\RnRnRn; N)$ implies 
  \begin{align}\label{p36}
    &\sup_{\substack{ x,\tilde{x} \in \Rn \\ x \neq \tilde{x} } } \left\{ \frac{ | \p_{x}^{\alpha} \p_y^{\gamma} \p_{\eta}^{\beta} a(x, \xi + \eta, x_1 + y) - \p_{ \tilde{x} }^{\alpha} \p_y^{\gamma} \p_{\eta}^{\beta} a( \tilde{x}, \xi + \eta, x_1 + y) | }{|x - \tilde{x}|^s}  \right\} \notag\\
    & \qquad \leq \| \p_y^{\gamma} \p_{\eta}^{\beta} a(x, \xi + \eta, x_1 + y) \|_{ C^{\tilde{m}, s} (\Rn_{x}) } \leq C_{\gamma,\beta} \<{\xi + \eta}^m
  \end{align}
  for all $\xi, \eta,y \in \Rn$ and all $a \in \mathscr{B}$.
  By means of the fundamental theorem of calculus for $|x_1-x_2| <1$ and on account of $|x_1-x_2|^s \geq 1$ for $|x_1-x_2| \geq 1$ we obtain due to the boundedness of $\mathscr{B} \subseteq  C^{\tilde{m},s} S^{m}_{0,0}(\RnRnRn; N)$ 
  \begin{align}\label{p37}
    \frac{ | \p_{x}^{\alpha} \p_y^{\gamma} \p_{\eta}^{\beta} a(x, \xi + \eta, x_1 + y) - \p_{x}^{\alpha} \p_y^{\gamma} \p_{\eta}^{\beta} a(x, \xi + \eta, x_2 + y) | }{|x_1 - x_2|^s}
    \leq C_{\beta, \gamma} \<{\xi + \eta}^m
  \end{align}
  for all $a \in \mathscr{B}$ and $x,\xi, \eta,x_1, x_2,y \in \Rn$, $x_1 \neq x_2$.
  Finally, the proposition follows from (\ref{p36}) and (\ref{p37})
  by means of the triangle inequality. 
\end{proof}

\begin{lemma}\label{lemma:AblVonALStetig}
  Let $N \in \N_0 \cup \{ \infty \}$  with $N > n$. Moreover, we define $\tilde{N}:= N-(n+1)$. For $a \in C^{ \tilde{m} ,s} S^m_{\rho,0}(\RnRnRn; N)$ with $\tilde{m} \in \N_0$, $m \in \R$, $0 \leq  \rho \leq 1$ and $0<s<1$, we define $a_L:\RnRn \rightarrow \C$ as in Lemma \ref{lemma:AbschVonaL}.
  Then $ \p_x^{\delta} \pa{\gamma} a_L \in C^{0}(\RnRn)$ for every $\gamma, \delta \in \Non$ with $|\delta| \leq \tilde{m}$ and $|\gamma| \leq \tilde{N}$.
\end{lemma}

\begin{proof}
  Let $\alpha, \beta \in \Non$ with $|\beta| \leq \tilde{m}$ and $|\alpha| \leq \tilde{N}$. On account of Remark \ref{bem:AbleitungVonPDOWiederPDO} we know that
  $\p_x^{\beta} \pa{\alpha}  a  \in C^0(\RnRnRn)$. 
  With  $a(x, \xi + \eta, x+y) \in \mathscr{A}_0^{m^+, N}(\R^{2n}_{(y,y')} \times \R^{2n}_{(\xi,\eta)})$ at hand we are able to apply 
  Theorem \ref{thm:VertauschenVonOsziIntUndAbleitungen} and get together with Lemma \ref{lemma:VertauscheAblNachXUndOsziInt} for all $x,\xi \in \Rn$: 
  \begin{align} \label{eq5} 
    \p_x^{\beta} \pa{\alpha}  a_L  (x,\xi) 
    &= \p_x^{\beta} \osint e^{-iy \cdot \eta} \pa{\alpha} a (x, \eta + \xi, x+y) dy \dq \eta \notag \\
    &= \osint e^{-iy \cdot \eta} \p_x^{\beta} \{ \pa{\alpha} a (x, \eta + \xi, x+y)\} dy \dq \eta .
  \end{align}
  In order to show the continuity of $\p_x^{\beta} \pa{\alpha}  a_L $, we want to apply Corollary \ref{kor:VertauschbarkeitVonOsziIntUndLimes}. 
  To this end let $(x, \xi) \in \RnRn$ be arbitrary. Additionally let $(x', \xi') \in \RnRn$ with $|x-x'|, |\xi - \xi'| < 1 $. 
  For every $\beta_1, \beta_2, \gamma, \delta \in \Non$ with $\beta_1 + \beta_2 = \beta$ and $|\delta| \leq N-|\alpha|$ an application of $a \in C^{ \tilde{m} ,s} S^m_{\rho,0}(\RnRnRn; N)$  provides 
  \begin{align*}
    &|\p_y^{\gamma} \p_{\eta}^{\delta} (\p_x^{\beta_1} \pa{\alpha} \p_y^{\beta_2}a) (x', \eta + \xi', x'+y)| 
    \leq C_{\alpha, \beta, \gamma, \delta} \<{\eta + \xi' }^{ m - \rho (|\alpha| + |\delta|)} \\
    &\qquad \leq C_{\alpha, \beta, \gamma, \delta} \<{\eta}^{m} \<{ \xi' }^{ |m| }
     \leq C_{\alpha, \beta, \gamma, \delta} \<{\eta}^{m} \<{ \xi'- \xi }^{ |m| } \<{\xi}^{|m|}
    \leq C_{\alpha, \beta, \gamma, \delta} \<{\eta}^{m} \<{ \xi }^{ |m| }.
  \end{align*}
  Here $C_{\alpha, \beta, \gamma, \delta}$ is independent of $x', \eta, \xi', y \in \Rn$. This yields the boundedness of 
  \begin{align*}
    \{ (\p_x^{\beta_1} \pa{\alpha} \p_y^{\beta_2}a) (x', \eta + \xi', x'+y) : x', \xi' \in \Rn \text{ with } |x-x'|, |\xi - \xi'| < 1\} 
  \end{align*}
  in $\mathscr{A}_0^{m, N-|\alpha|}(\RnRnx{y}{\eta})$. Moreover, we obtain for all $y, \eta \in \Rn$ and for each $ \beta_1, \beta_2, \gamma, \delta  \in \Non$ with $\beta_1 + \beta_2 = \beta$ and $|\delta| \leq N-|\alpha|$:
  \begin{align*}
    \p_y^{\gamma} \p_{\eta}^{\delta} (\p_x^{\beta_1} \pa{\alpha} \p_y^{\beta_2}a) (x', \eta + \xi', x'+y) \xrightarrow[x' \rightarrow x]{\xi' \rightarrow \xi} 
    \p_y^{\gamma} \p_{\eta}^{\delta} (\p_x^{\beta_1} \pa{\alpha} \p_y^{\beta_2}a) (x, \eta + \xi, x+y)
  \end{align*}
  due to $a \in C^{ \tilde{m} ,s} S^m_{\rho,0}(\RnRnRn; N)$.
  Using Leibniz's rule and Corollary \ref{kor:VertauschbarkeitVonOsziIntUndLimes} yields 
  \begin{align*}
    &\lim_{\substack{\xi' \rightarrow \xi \\ x' \rightarrow x} } \osint e^{-iy \cdot \eta} \p_{x'}^{\beta} \{ \pa{\alpha} a (x', \eta + \xi', x'+y)\} dy \dq \eta\\
    & \qquad \qquad \qquad = \osint e^{-iy \cdot \eta} \p_x^{\beta} \{ \pa{\alpha} a (x, \eta + \xi, x+y)\} dy \dq \eta.
  \end{align*}
  Hence $\p_x^{\beta} \pa{\alpha}  a_L $ is continuous.
\end{proof}

Now we are in the position to show that $a_L$ is a non-smooth symbol.
Unfortunately we loose some smoothness with respect to $\xi$ of the double symbol:

\begin{thm}\label{thm:aLInCsSm00}
  Let $0 < s < 1$, $\tilde{m} \in \N_0$ and $m \in \R$. Additionally we choose $N \in \N_0 \cup \{ \infty \}$ with $N > n$. We define $\tilde{N}:= N-(n+1)$. Furthermore, let $\mathscr{B} \subseteq C^{\tilde{m}, s}_*S^{m}_{0,0}(\RnRnRn; N)$ be bounded. If we define for each $a \in \mathscr{B}$ the function $a_L: \RnRn \rightarrow \C$ as in Lemma \ref{lemma:AbschVonaL},
  we get $a_L \in  C^{\tilde{m},s}S^{m}_{0,0}(\RnRn; \tilde{N}) $ for all $a \in \mathscr{B}$ and the existence of a constant $C_{\beta}$, independent of $a \in \mathscr{B}$, such that
  \begin{align*}
    \| \pa{\beta} a_L(.,\xi) \|_{ C^{\tilde{m}, s} (\Rn) } \leq C_{\beta} \<{\xi}^m \qquad \text{for all } \xi \in \Rn \text{ and } \beta \in \Non \text{ with } |\beta| \leq \tilde{N}.
  \end{align*}
  This implies the boundedness of $\{ a_L: a \in \mathscr{B} \} \subseteq  C^{\tilde{m},s} S^{m}_{0,0}(\RnRn; \tilde{N}) $.
\end{thm}

\begin{proof}
  Due to Lemma \ref{lemma:AblVonALStetig} we have $\p_x^{\delta} \pa{\gamma} a_L \in C^{0}(\RnRn)$ for every $\gamma, \delta \in \Non$ with $|\gamma| \leq \tilde{N}$ and $|\delta| \leq \tilde{m}$. Since  $a(x, \xi + \eta, x+y)$ is an element of $ \mathscr{A}_0^{m^+, N}(\R^{2n}_{(y,y')} \times \R^{2n}_{(\xi,\eta)})$ and $N-|\alpha| > n$, we derive from Theorem \ref{thm:VertauschenVonOsziIntUndAbleitungen} for each $\alpha \in \Non$ with $|\alpha| \leq \tilde{N}$:
  \begin{align*}
    \pa{\alpha} a_L(x,\xi) = \osint e^{-iy \cdot \eta} \pa{\alpha} a(x, \eta + \xi, x+y) dy \dq \eta \qquad \text{for all } x, \xi \in \Rn.
  \end{align*}
  Let $\alpha \in \Non$ with $|\alpha| \leq \tilde{N}$. Remark \ref{bem:AbleitungVonPDOWiederPDO} and the boundedness of $\mathscr{B}$ implies the boundedness of 
  $$\mathscr{\tilde{B}} := \left\{ \pa{\alpha} a : a \in \mathscr{B} \right\} \subseteq C^{\tilde{m},s} S^{m}_{0, 0}(\RnRn \times \Rn, N-|\alpha|).$$
  Hence it remains to show
  \begin{align}\label{p33}
    \| a_L(.,\xi) \|_{C^{\tilde{m},s}(\Rn)} \leq C \<{\xi}^m \quad \text{for all }  \xi \in \Rn, a \in \mathscr{\tilde{B}}.
  \end{align}
  Inequality (\ref{p33}) implies $\| \pa{\alpha} a_L(.,\xi) \|_{C^{\tilde{m},s}(\Rn)} \leq C_{\alpha} \<{\xi}^m$
 for all $\xi \in \Rn$ and $a \in \mathscr{B}$. This yields the boundedness of $\{ a_L : a \in \mathscr{B} \} \subseteq C^{\tilde{m},s} S^m_{0,0}(\RnRnRn; \tilde{N} )$.
  Now we choose $l \in \N_0$ with $-2l+|m|<-n$ and $l_0:= N-\tilde{N}$. An application of Lemma \ref{lemma:VertauscheAblNachXUndOsziInt} and Theorem \ref{thm:ExistenceOfOscillatoryIntegral} provides for every $\delta \in \Non$ with $|\delta| \leq \tilde{m}$:
  \begin{align}\label{p34}
    &\p_x^{\delta} a_L(x,\xi) = \osint e^{-iy \cdot \eta} \p_x^{\delta} \left\{ a(x, \eta + \xi, x+y) \right\} dy \dq \eta \notag \\
    &\quad = \iint e^{-iy \cdot \eta} \<{\eta}^{-2l} \<{ D_{y} }^{2l} \left\{ A^{l_0}(D_{\eta},y) \p_x^{\delta} [ a(x, \xi + \eta, x + y ) ] \right\} dy \dq \eta.
  \end{align}
  On account Proposition \ref{prop:fürAbschVonHoelderNorm} and $\left| \p_y^{\alpha_1} \frac{y_j}{\<{y}} \right| \leq 1$ for all $j \in \{ 1,\ldots, n\}$ and $\alpha_1 \in \Non$ we obtain similary to the proof of Remark \ref{bem:Absch} for
  $\delta \in \Non$ with $|\delta| \leq \tilde{m}$:
  \begin{align*}
    &\left| \<{\eta}^{-2l} \<{D_y}^{2l} A^{l_0}(D_{\eta}, y) \left\{ \frac{ \p_{x_1}^{\delta} a(x_1, \xi + \eta, x_1 + y) - \p_{x_2}^{\delta} a(x_2, \xi + \eta, x_2 + y) }{(x_1-x_2)^s} \right\} \right| \\
    &\qquad 
    \leq C \<{y}^{-l_0} \<{\xi}^m \<{\eta}^{-2l+|m|}
  \end{align*}
  for all $x_1, x_2, y, \xi, \eta \in \Rn$ with $x_1 \neq x_2$ and all $a \in \mathscr{\tilde{B}}$.  
  Consequently we have for each $\delta \in \Non$ with $|\delta| \leq \tilde{m}$:
  \begin{align*}
    \frac{| \p_x^{\delta} a_L(x_1,\xi) - \p_x^{\delta} a_L(x_2,\xi) |}{|x_1 - x_2|^s} 
    \leq \iint C \<{y}^{-l_0} \<{\xi}^m \<{\eta}^{-2l+|m|} dy \dq \eta \leq C \<{\xi}^m
  \end{align*}
  for all $x_1,x_2, \xi \in \Rn$ with $x_1 \neq x_2$ and $a \in \mathscr{\tilde{B}}$.
  Finally, we only have to use the previous inequality and Lemma \ref{lemma:AbschVonaL} to get (\ref{p33}).
\end{proof}

We still need to show $a_L(x,D_x)=a(x,D_x,x')$. For this we need:

\begin{prop}\label{prop:PartIntFubiniLebesgueFuerReduktion}
  Let $\tilde{m} \in \N_0$, $m \in \R$, $0<s<1$ and $\chi \in \snn$. Additionally let $N \in \N_0 \cup \{ \infty \}$ with $N > n$. Moreover, we choose $l,l_0,l_0' \in \N_0$ with
  \begin{align*}
      -2l + m < -n,   \qquad   -2l_0  <-n, \qquad -2l_0' + 2l <-n.
  \end{align*}
  Assuming $0< \e' < 1$, $a \in C^{ \tilde{m} ,s} S^m_{0,0}(\RnRnRn; N)$ and $u \in \s$ we define for every $0 < \e < 1$ the functions $a_0, \hat{a}, a_{\e}, \tilde{a}_{0}: \R^{5n} \rightarrow \C$ by
  \begin{align*}
    a_0 (x,x', x'', \xi, \xi') &:= \chi(\e' x'', \e' \xi') a(x,\xi, x') u(x'')\\
    \hat{a}(x,x', x'', \xi, \xi') &:= \<{x'-x''}^{-2l_0} \<{ D_{\xi'} }^{2l_0} \left[ \<{\xi'}^{-2l_0'} \<{D_{x''} }^{2l_0'} a_0 (x,x', x'', \xi, \xi') \right], \\
    a_{\e}(x,x', x'', \xi, \xi') &:= \chi(\e x', \e \xi) \hat{a}(x,x', x'', \xi, \xi'), \\
    \tilde{a}_0 (x,x', x'', \xi, \xi') &:= \<{- \xi + \xi'}^{-2l} \<{D_{x'} }^{2l} \hat{a} (x,x', x'', \xi, \xi').
  \end{align*}
	for all $x,x', x'', \xi, \xi' \in \Rn$. Then 
  \begin{align*}
    &\lim_{\e \rightarrow 0} \iiiint e^{-ix'' \cdot \xi'} e^{-ix'\cdot \xi + ix'\cdot \xi' + ix\cdot \xi}  a_{\e} (x,x', x'', \xi, \xi') d x'' \dq \xi' dx' \dq \xi  \\
    &\qquad = \iiiint e^{-ix'' \cdot \xi'} e^{-ix'\cdot \xi + ix'\cdot \xi' + ix\cdot \xi}  \tilde{a}_{0} (x,x', x'', \xi, \xi') d x'' \dq \xi' dx' \dq \xi. 
  \end{align*}

\end{prop}

\begin{proof}
  First of all we define for each $0 < \e < 1$ the function $\tilde{a}_{\e}: \R^{5n} \rightarrow \C$ by
  \begin{align*}
  	\tilde{a}_{\e} (x,x', x'', \xi, \xi') &:=  \<{- \xi + \xi'}^{-2l} \<{D_{x'} }^{2l} a_{\e}(x,x', x'', \xi, \xi') 
  \end{align*}
  for all $x,x', x'', \xi, \xi' \in \Rn$. By means of the Leibniz rule, $ a \in C^{ \tilde{m} ,s} S^m_{0,0}(\RnRnRn; N)$, $u \in \s$ and  $\chi \in \snn$ one can show for a fixed $x \in \Rn$ and for arbitrary $M, M_1, M_2 \in \N_0$ with $-2l_0 - M_1 < -n$ and $m-M_2 < -n$: 
  \begin{align}
    \left| a_{\e} (x,x', x'', \xi, \xi') \right| &\leq C |\chi(\e x', \e \xi)| \<{x'-x''}^{-2l_0} \<{\xi'}^{-2l_0'} \<{\xi}^{m}\<{x''}^{-M} \notag \\
    &\leq C_{\e} \<{x'}^{-2l_0-M_1} \<{\xi}^{m-M_2} \<{x''}^{2l_0-M} \<{\xi'}^{-2l_0'}, \label{p46} \\
    \left|  \tilde{a}_{\e} (x,x', x'', \xi, \xi') \right| &\leq C \<{-\xi + \xi'}^{-2l} \<{x'-x''}^{-2l_0} \<{\xi'}^{-2l_0'} \<{\xi}^{m} \<{x''}^{-M} \notag \\
    &\leq C \<{x'}^{-2l_0} \<{x''}^{2l_0-M} \<{\xi'}^{-2l_0' + 2l} \<{\xi}^{-2l + m}, \label{p47}\\
    \left|  \tilde{a}_{0} (x,x', x'', \xi, \xi') \right| &\leq C \<{-\xi + \xi'}^{-2l} \<{x'-x''}^{-2l_0} \<{\xi'}^{-2l_0'} \<{\xi}^{m} \<{x''}^{-M} \notag \\
    &\leq C \<{x'}^{-2l_0} \<{x''}^{2l_0-M} \<{\xi'}^{-2l_0' + 2l} \<{\xi}^{-2l + m}, \label{p51}
  \end{align}
  where $C$ is independent of $x,x', x'', \xi, \xi' \in \Rn$ and of $0 < \e < 1$. 
  Now we choose $M \in \N$ with $2l_0-M <-n$. Then we have 
  $a_{\e} (x,x', x'', \xi, \xi') \in L^1(\RnRnx{x''}{\xi'} \times \RnRnx{x'}{\xi})$ and $ \tilde{a}_{\e} (x,x', x'', \xi, \xi')  \in L^1(\RnRnx{x'}{\xi})$ for every fixed $x, x'', \xi' \in \Rn$ and $0 < \e < 1$.  Hence we are able to use Fubini's theorem first and integrate by parts with respect to $x'$ and $\xi$ afterwards and get
  \begin{align}\label{eq16}
    &\iiiint e^{-ix'' \cdot \xi'} e^{-ix'\cdot \xi + ix'\cdot \xi' + ix\cdot \xi}  a_{\e} (x,x', x'', \xi, \xi') d x'' \dq \xi' dx' \dq \xi \notag \\
    &\qquad = \iiiint e^{-ix'' \cdot \xi'} e^{-ix'\cdot \xi + ix'\cdot \xi' + ix\cdot \xi}  a_{\e} (x,x', x'', \xi, \xi') dx' \dq \xi d x'' \dq \xi' \notag\\
    &\qquad = \iiiint e^{-ix'' \cdot \xi'} e^{-ix'\cdot \xi + ix'\cdot \xi' + ix\cdot \xi}  \tilde{a}_{\e} (x,x', x'', \xi, \xi') dx' \dq \xi d x'' \dq \xi'.
  \end{align}
   Because of (\ref{p47}) we have 
   for every $x \in \Rn$ 
  \begin{align}\label{p49}
  	&|e^{-ix'' \cdot \xi'-ix'\cdot \xi + ix'\cdot \xi' + ix\cdot \xi}  \tilde{a}_{\e} (x,x', x'', \xi, \xi')| \notag \\
	&\qquad \qquad \qquad \qquad \qquad \leq C_x \<{x'}^{-2l_0} \<{x''}^{2l_0-M} \<{\xi'}^{-2l_0' + 2l} \<{\xi}^{-2l + m} \notag\\
  	&\qquad \qquad \qquad \qquad \qquad \in L^1(\RnRnx{x'}{\xi} \times \RnRnx{x''}{\xi'} ),
  \end{align}
  where $C_x$ is independent of $x', x'', \xi, \xi' \in \Rn$ and of $0 < \e < 1$. Making use of Fubini's theorem in (\ref{eq16}) provides:
  \begin{align}\label{p50}
  	&\iiiint e^{-ix'' \cdot \xi'} e^{-ix'\cdot \xi + ix'\cdot \xi' + ix\cdot \xi}  a_{\e} (x,x', x'', \xi, \xi') d x'' \dq \xi' dx' \dq \xi \notag\\
    &\qquad = \iiiint e^{-ix'' \cdot \xi'} e^{-ix'\cdot \xi + ix'\cdot \xi' + ix\cdot \xi}  \tilde{a}_{\e} (x,x', x'', \xi, \xi') d x'' \dq \xi' dx' \dq \xi .
  \end{align}
  It remains to calculate the limit $\e \rightarrow 0$. Due to (\ref{p51}) we know that the function
  $e^{-ix'' \cdot \xi'} e^{-ix'\cdot \xi + ix'\cdot \xi' + ix\cdot \xi} \tilde{a}_0 (x,x', x'', \xi, \xi')$ is an element of $L^1(\RnRnx{x'}{\xi} \times \RnRnx{x''}{\xi'} )$. Using the definition of $\<{ D_{x'} }^{2l}$ and the Leibniz rule, one easily obtains by the pointwise convergence of $\chi_{\e}$ to $1$ and the pointwise convergence of every derivative of $\chi_{\e}$ to $0$, see e.g. \cite[Lemma 6.3]{KumanoGo}:
  \begin{align*}
  	e^{-ix'' \cdot \xi'} e^{-ix'\cdot \xi + ix'\cdot \xi' + ix\cdot \xi} \tilde{a}_{\e} (x,x', x'', \xi, \xi') \rightarrow
  	e^{-ix'' \cdot \xi'} e^{-ix'\cdot \xi + ix'\cdot \xi' + ix\cdot \xi} \tilde{a}_0 (x,x', x'', \xi, \xi')
  \end{align*} 
  for all $x,x', x'', \xi, \xi' \in \Rn$ if $\e \rightarrow 0$. 
  We conclude the claim by applying Lebesgue's theorem to (\ref{p50}) which is possible due to the previous convergence and (\ref{p49}).
\end{proof}

Combining the previous results we obtain:

\begin{thm}\label{thm:SymbolReduktionNichtGlatt}
  Let $N \in \N_0 \cup \{ \infty \}$ with $N > n$. We define $\tilde{N}:= N-(n+1)$. Assuming an $a \in C^{ \tilde{m} ,s} S^m_{0,0}(\RnRnRn; N)$ with $\tilde{m} \in \N_0$, $m \in \R$ and $0<s<1$, we define $a_L:\RnRn \rightarrow \C$ by
  \begin{align*}
    a_L(x,\xi) := \osint e^{-iy \cdot \eta} a(x, \eta + \xi, x+y) dy \dq \eta \in C^{\tilde{m}, s} S^m_{0,0} (\RnRnx{x}{\xi}; \tilde{N})
  \end{align*}
  for all $x, \xi \in \Rn$. Then we have for every $u \in \s$
  \begin{align*}
    a(x,D_x, x') u = a_L(x,D_x) u.
  \end{align*}
\end{thm}

\begin{proof}
  We already know that $a_L \in C^{\tilde{m}, s} S^m_{0,0} (\RnRn; \tilde{N})$ due to Theorem \ref{thm:aLInCsSm00}. Now we choose $u \in \s$ and $l,l_0,l_0' \in \Non$ with the property
\begin{align}
    \begin{array}{ccccc}
      -2l+ m < -n, \quad & \quad -2l_0 <-n, \quad &  \quad & \quad -2l_0' + 2l <-n.
    \end{array}
  \end{align}
  On account of $a_L(x,\xi')u(x'') \in \mathscr{A}^{m, \tilde{N}}_{-k}(\RnRnx{x''}{\xi'})$ for all $k \in \N_0$, Theorem \ref{thm:ExistenceOfOscillatoryIntegral} yields the existence of $a_L(x,D_x)u$. Because of $a(x, ,\eta + \xi', x+y) \in  \mathscr{A}^{m, N}_{0}(\RnRnx{y}{\eta})$ for every fixed $x, \xi' \in \Rn$ we can apply Theorem \ref{thm:VariableTransformationOfOsiInt} and get
  \begin{align} \label{eq12} %\label{p52}
    &a_L(x,D_x)u(x) 
    = \osint e^{i(x-x'')\cdot \xi'} \osint e^{-iy \cdot \eta} a(x, \eta + \xi', x+y) u(x'') dy \dq \eta d x'' \dq \xi'  \notag\\
    &\quad = \osint e^{i(x-x'')\cdot \xi'} \osint e^{-i(x'-x) \cdot (\xi-\xi')} a(x, \xi, x') u(x'') dx' \dq \xi d x'' \dq \xi'\notag\\ 
    &=\lim_{\e' \rightarrow 0} \iint e^{i(x-x'')\cdot \xi'} \chi(\e' x'', \e' \xi')  \notag\\
    &\qquad \qquad \lim_{\e \rightarrow 0} \iint  e^{-i(x'-x) \cdot (\xi-\xi')} a(x, \xi, x') \chi(\e x', \e \xi)  u(x'') dx' \dq \xi d x'' \dq \xi',
  \end{align}
  where $\chi \in \snn$ with $\chi(0,0)=1$.
  Integration by parts yields for arbitrary $0< \e < 1$ and $k, k' \in \N_0$ with $-N \leq -k<-n$ and $-2k'+m < -n$ on account of Remark \ref{bem:eFunktion}:
  \begin{align}\label{eq9}
    &\iint  e^{-i(x'-x) \cdot (\xi-\xi')} a(x, \xi, x') \chi(\e x', \e \xi)  u(x'') dx' \dq \xi \notag \\
    & \qquad = \iint  e^{-i(x'-x) \cdot (\xi-\xi')} b_{\e}(x,\xi,x',x'') dx' \dq \xi,
  \end{align}
  where 
  $$b_{\e}(x,\xi,x', \xi',x'') \hspace{-0.7mm}:= \hspace{-0.7mm} A^k(D_{\xi}, x'-x) \left[ \hspace{-0.4mm} \<{\xi - \xi'}^{-2k'} \<{D_{x'} }^{2k'} \chi(\e x', \e \xi) a(x,\xi, x') u(x'') \right]\hspace{-0.7mm}$$
  for all $x, \xi, x', \xi', x'' \in \Rn$ and each $0 \leq \e < 1$.  
  We choose $M_1, M_2 \in \N$ with $-M_2 < -2n$ and $-M_1 + M_2 < -n$. Using Leibniz's rule and Petree's inequality
  we obtain 
  for arbitrary but fixed $x, \xi', x'' \in \Rn$ if we use Remark \ref{bem:Absch}:
  \begin{align}\label{eq8}
    |b_{\e}(x,\xi,x', \xi',x'')| 
    &\leq C  \<{x'-x}^{-k}  \<{\xi - \xi'}^{-2k'} \<{\xi}^m \<{x''}^{-M_1} \notag \\
    &\leq C_x  \<{x'}^{-k}  \<{\xi'}^{2k'} \<{\xi}^{m-2k'} \<{x''}^{-M_1}
    \in L^1(\RnRnx{x'}{\xi}),
  \end{align}
  where $C_x$ is independent of $0< \e < 1$ and $x', \xi, \xi', x'' \in \Rn$.
  Since $b_{\e}(x,\xi,x', \xi',x'')$ converges to $b_{0}(x,\xi,x', \xi',x'')$ as $\e \rightarrow 0$ and (\ref{eq8}) holds we are able to apply Lebesgue's theorem to (\ref{eq9}) and get 
  \begin{align*}
    &e^{i(x-x'')\cdot \xi'} \chi(\e' x'', \e' \xi')  \iint  e^{-i(x'-x) \cdot (\xi-\xi')} a(x, \xi, x') \chi(\e x', \e \xi)  u(x'') dx' \dq \xi \notag\\
    &\qquad  \xrightarrow[]{\e \rightarrow 0} e^{i(x-x'')\cdot \xi'} \chi(\e' x'', \e' \xi') \iint e^{-i(x'-x) \cdot (\xi-\xi')} b_{0}(x,\xi,x', \xi',x'')  dx' \dq \xi
  \end{align*}
  for every $x, x'', \xi' \in \Rn$ and $0< \e' < 1$. 
  Additionally using (\ref{eq9}), (\ref{eq8}) and $\chi \in \snn$
  yields for fixed but arbitrary $0< \e' <1$:
  \begin{align*}
    &\left| e^{i(x-x'')\cdot \xi'} \chi(\e' x'', \e' \xi') \iint  e^{-i(x'-x) \cdot (\xi-\xi')} a(x, \xi, x') \chi(\e x', \e \xi)  u(x'') dx' \dq \xi \right| \notag \\
    &\qquad \leq C_{\e'} \iint \left| \<{(x'', \xi')}^{-M_2-2k'} e^{-i(x'-x) \cdot (\xi-\xi')} b_{\e}(x,\xi,x', \xi',x'') \right| dx' \dq \xi \notag \\
    &\qquad \leq C_{\e'} \iint \<{x'}^{-k} \<{\xi}^{m-2k'} \<{(x'', \xi')}^{-M_2} dx' \dq \xi \notag\\
    &\qquad \leq C_{\e'} \<{(x'', \xi')}^{-M_2}
    \in L^1(\RnRnx{x''}{\xi'}).
  \end{align*}
  Applying Lebesgue's theorem we obtain because of (\ref{eq12}):
  \begin{align}\label{eq13}
    a_L(x,D_x)u(x) &= \lim_{\e' \rightarrow 0} \lim_{\e \rightarrow 0} \iint e^{i(x-x'')\cdot \xi'} \chi(\e' x'', \e' \xi')  \notag \\
    & \qquad \iint  e^{-i(x'-x) \cdot (\xi-\xi')} a(x, \xi, x') \chi(\e x', \e \xi)  u(x'') dx' \dq \xi d x'' \dq \xi' \notag \\
    &= \lim_{\e' \rightarrow 0} \lim_{\e \rightarrow 0} \iiiint  e^{i(x-x'')\cdot \xi'} \chi(\e' x'', \e' \xi')  e^{-i(x'-x) \cdot (\xi-\xi')} a(x, \xi, x') \notag \\
    & \qquad \qquad \qquad \qquad \qquad   \chi(\e x', \e \xi)  u(x'') d x'' \dq \xi' dx' \dq \xi.
  \end{align}
  Now we define 
  \begin{align*}
    \tilde{a}_{\e'}(x,x', x'', \xi, \xi') &:= \chi(\e' x'', \e' \xi') a(x,\xi, x') u(x'') , \\
    a_{\e'}(x,x', x'', \xi, \xi') &:= \<{x'-x''}^{-2l_0} \<{ D_{\xi'} }^{2l_0} \left[ \<{\xi'}^{-2l_0'} \<{D_{x''} }^{2l_0'} \tilde{a}_{\e'}(x,x', x'', \xi, \xi') \right], \\
    \hat{a}_{\e'}(x,x', x'', \xi, \xi') &:= \<{-\xi + \xi'}^{-2l} \<{D_{x'} }^{2l}  a_{\e'}(x,x', x'', \xi, \xi') 
  \end{align*}
  for all $x,x', x'', \xi, \xi' \in \Rn$ and $0< \e < 1$. 
  Integrating by parts in (\ref{eq13}) provides:
  \begin{align}\label{eq15}
     &a_L(x,D_x)u(x) \notag\\ 
    &\quad = \lim_{\e' \rightarrow 0} \lim_{\e \rightarrow 0} \iiiint  e^{i(x-x'')\cdot \xi' -i(x'-x) \cdot (\xi-\xi')} \chi(\e x', \e \xi) a_{\e'}(x,x', x'', \xi, \xi')  d x'' \dq \xi' dx' \dq \xi \notag \\
    &\quad= \lim_{\e' \rightarrow 0}  \iiiint  e^{-ix''\cdot \xi' -i(x'-x) \cdot \xi + ix' \cdot \xi'} \hat{a}_{\e'}(x,x', x'', \xi, \xi')  d x'' \dq \xi' dx' \dq \xi
  \end{align}
  due to Proposition \ref{prop:PartIntFubiniLebesgueFuerReduktion}. We define
  \begin{align*}
    \hat{a}(x,x', x'', \xi, \xi') &:=  \<{-\xi + \xi'}^{-2l} \<{D_{x'} }^{2l}  a_0(x,x', x'', \xi, \xi'), \\
    a_0 (x,x', x'', \xi, \xi') &:= \<{x'-x''}^{-2l_0} \<{ D_{\xi'} }^{2l_0} \left[ \<{\xi'}^{-2l_0'} \<{D_{x''} }^{2l_0'}  a(x,\xi, x') u(x'') \right]
  \end{align*}
  for all $x,x', x'', \xi, \xi' \in \Rn$. Then
  \begin{align*}
    \hat{a}_{\e'}(x,x', x'', \xi, \xi') \xrightarrow[]{ \e' \rightarrow 0}   
    \hat{a}(x,x', x'', \xi, \xi')
  \end{align*}
  for all $x,x', x'', \xi, \xi' \in \Rn$. Similarly to (\ref{eq8}) we get due to Leibniz's rule and Petree's inequality:
  \begin{align*}
    |\hat{a}_{\e'}(x,x', x'', \xi, \xi')| 
    \leq C \<{x'}^{ - 2l_0} \<{x''}^{2l_0-M} \<{\xi}^{m-2l} \<{\xi'}^{2l-2l_0'} \in L^1(\Rn_{x'}\times \Rn_{\xi} \times \Rn_{x''} \times \Rn_{\xi'}).
  \end{align*}
  Here the constant $C$ is independent of $0< \e' <1$ and $x, x', x'', \xi, \xi' \in \Rn$. 
  An application of Lebesgue's theorem to (\ref{eq15}) provides:
  \begin{align*}
    a_L(x,D_x)u(x) 
    =\iiiint  e^{-ix''\cdot \xi' -i(x'-x) \cdot \xi + ix' \cdot \xi'}  \hat{a}(x,x', x'', \xi, \xi')  d x'' \dq \xi' dx' \dq \xi.
  \end{align*}
  Hence we get the claim by using Lemma \ref{lemma:SymbolReduction}.
\end{proof}

%% file: PropertiesOfTepsilon.tex
\subsection{Properties of the Operator $T_{\e}$}\label{section:PropertiesOfTe}

Besides the results of Subsection \ref{section:pointwiseConvergence} and Subsection \ref{section:symbolReduction} we need an approximation $( T_{\e} )_{\e \in (0,1]}$ for a given operator $T \in \mathcal{A}^{0, M}_{0,0}(\tilde{m},q)$ such that
\begin{itemize}
  \item $T_{\e}: \sd \rightarrow \s$ is continuous,
  \item The iterated commutators of $T_{\e}$ are uniformly bounded with respect to $\e$ as maps from $L^q(\Rn)$ to $L^q(\Rn)$, 
  \item $T_{\e}$ converges pointwise to $T$ if $\e \rightarrow 0$.
\end{itemize}

Throughout this subsection we assume: Let $1 < q < \infty$ and $M \in \N_0 \cup \{ \infty \}$ be arbitrary. Moreover, let $T \in \mathcal{A}^{0, M}_{0,0}(\tilde{m},q)$ with $\tilde{m} \in \N_0$ and $\varphi \in \con$ with $\varphi(x)=1$ for all $|x|\leq \frac{1}{2}$ and  $\varphi(x)=0$ for all $|x|\geq 1$. Then we define for $0< \e \leq 1$ 
  \begin{align*}
    P_{\e} := \tilde{p}_{\e}(x,D_x) \qquad \text{ and } \qquad Q_{\e} := q_{\e}(x,D_x), 
  \end{align*}
  where  the symbols 
  $\tilde{p}_{\e}$ and $q_{\e}$ are defined as $\tilde{p}_{\e}(x,\xi):=\varphi(\e x)$ and $q_{\e}(x,\xi):=\varphi(\e \xi)$. Then $\{p_{\varepsilon}| 0 < \varepsilon \leq 1 \} $ and $\{q_{\varepsilon}|0 < \varepsilon \leq 1 \} $ are bounded subsets of $S^0_{1, 0}(\RnRn)$, cf. e.g. \cite[Lemma 3.4]{Diss} for more details. Note that $u \in \s$ we have $P_{\e} u= \tilde{p}_{\e} u$.
  Additionally the continuity of multiplication operators with $C^{\infty}_c$-functions imply the continuity of $P_{\e} : C^{\infty}(\Rn) \rightarrow C^{\infty}_c(\Rn)$. Moreover, we define 
  $$T_{\e} := P_{\e} Q_{\e} T  P_{\e} Q_{\e}.$$

For the following we need:

\begin{lemma}\label{lemma:LqLimitOfTepsilonu}
	For all $u \in L^q(\Rn)$ we have the following convergence:
	\begin{align*}
		L^q- \lim_{\e \rightarrow 0} T_{\e}u = Tu.
	\end{align*}
\end{lemma}

\begin{proof}
  With the theorem of Banach-Steinhaus at hand, we easily can show
  \begin{align}\label{eq36}
    Q_{\e}u \xrightarrow[]{\e \rightarrow 0} u \qquad \text{and} \qquad P_{\e}u \xrightarrow[]{\e \rightarrow 0} u \qquad \text{in } L^q(\Rn).
  \end{align}
  For more details, see \cite[Proof of Lemma 5.27]{Diss}. 
  By means of (\ref{eq36}) and Theorem \ref{thm:stetigInBesselPotRaum}  we get for all $u \in L^q(\Rn)$:
  \begin{align*}
    \| P_{\e} Q_{\e}u - u \|_{L^q(\Rn)} 
    \leq C \| Q_{\e}u - u \|_{L^q(\Rn)} + \| P_{\e}u - u \|_{L^q(\Rn)} \xrightarrow[]{\e \rightarrow 0} 0.
  \end{align*}
  An application of Theorem \ref{thm:stetigInBesselPotRaum} gives us for all $u \in L^q(\Rn)$:
  \begin{align*}
    \| T_{\e} u - T u \|_{L^q(\Rn)} 
    &\leq \| P_{\e} Q_{\e} T P_{\e} Q_{\e} u - P_{\e} Q_{\e} T u \|_{L^q(\Rn)} + \| P_{\e} Q_{\e} T  u - T u \|_{L^q(\Rn)} \\
    &\leq C \| P_{\e} Q_{\e} u - u \|_{L^q(\Rn)} + \| P_{\e} Q_{\e} T  u - T u \|_{L^q(\Rn)} \xrightarrow[]{\e \rightarrow 0} 0.
  \end{align*}
  \vspace*{-1cm}

\end{proof}

\begin{lemma}\label{lemma:ContinuityOfIteratedCommutatorOfTeFromLqToLq}
	Let $\alpha, \beta \in \Non$ with $|\beta| \leq \tilde{m}$ and $|\alpha| \leq M$. Then
	\begin{align*}
		\| \ad(-ix)^{\alpha} \ad(D_x)^{\beta} T_{\e} \|_{ \mathscr{L}(L^q(\Rn)) } \leq C_{\alpha, \beta} \qquad \text{for all } 0 < \e \leq 1.
	\end{align*}
\end{lemma}

\begin{proof}
  Let $\alpha, \beta \in \Non$ with  $|\beta| \leq \tilde{m}$ and $|\alpha| \leq M$. 
  We define
  \begin{align*}
    R^{\beta_1, \beta_2, \beta_3}_{\alpha_1, \alpha_2, \alpha_3}:= \left[ \ad(D_{x})^{\beta_1} P_{\e} \right] \left[ \ad(-ix)^{\alpha_1} Q_{\e} \right] T^{\alpha_2,\beta_2} \left[ \ad(D_{x})^{\beta_3} P_{\e} \right] \left[ \ad(-ix)^{\alpha_3} Q_{\e} \right],
  \end{align*}
  where $T^{\alpha_2,\beta_2}:= \ad(-ix)^{\alpha_2}\ad(D_{x})^{\beta_2} T$.
  Then 
  we obtain for all $u \in \s$
  \begin{align*}
    \ad(-ix)^{\alpha} \ad(D_x)^{\beta} T_{\e} u = \sum_{ \substack{ \alpha_1 + \alpha_2 + \alpha_3 = \alpha \\ \beta_1 + \beta_2 + \beta_3 = \beta} } C_{\alpha_1, \alpha_2, \beta_1, \beta_2} R^{\beta_1, \beta_2, \beta_3}_{\alpha_1, \alpha_2, \alpha_3} u. 
  \end{align*}
  Due to Remark \ref{bem:SymbolOfIteratedCommutator} we get $\ad(D_{x})^{\gamma} P_{\e} \in \op S^0_{1,0}$ and $\ad(-ix)^{\delta} Q_{\e} \in \op S^{-|\delta|}_{1,0} \subseteq \op S^0_{1,0}$ for each $\gamma, \delta \in \Non$. On account of Theorem \ref{thm:stetigInBesselPotRaum}, the boundedness of  $\{p_{\varepsilon}| 0 < \varepsilon \leq 1 \} $ and $\{q_{\varepsilon}|0 < \varepsilon \leq 1 \} $ in $S^0_{1, 0}(\RnRn)$ and of $T \in \mathcal{A}^{0,M}_{0,0}(\tilde{m},q)$ we obtain 
  \begin{align*}
  	\| \ad(-ix)^{\alpha} \ad(D_x)^{\beta} T_{\e} u \|_{L^q} \leq C_{\alpha, \beta, q} \| u \|_{L^q} \qquad \text{for all } u \in \s.
  \end{align*}
  \vspace*{-1cm}

\end{proof}

\begin{prop}\label{prop:SmoothnessOfPepsilon0}
  Let $g \in \s$ and $0 < \e \leq 1$. For each $y \in \Rn$ we define $g_y:= \tau_y(g)$. Moreover, we define 
%   the function $p_{\e,0}: \RnRnRn \rightarrow \C$ as 
  $$p_{\e,0} (x,\xi,y):= e^{-ix \cdot \xi} T_{\e} ( e_{\xi} g_y)(x) \qquad \text{for all } (x,\xi,y) \in \RnRnRn.$$
  Then $p_{\e,0} \in C^{\infty}(\RnRnRn)$.
\end{prop}

In order to prove the previous proposition, we need:

\begin{Def}
  For $k \in \N_0$ we define the normed space $L^q_k(\Rn)$ as
  \begin{align*}
    L^q_k(\Rn):= \left\{ f\in L^q(\Rn) : \| f \|_{L^q_k} := \| \<{x}^{k+1} f(x) \|_{L^q(\Rn_x)} < \infty \right\}.
  \end{align*}
\end{Def}

\begin{proof}[Proof of Proposition \ref{prop:SmoothnessOfPepsilon0}:]
  Let $k \in \N_0$ be arbitrary but fixed. For every $x,\xi \in \Rn$, $f \in C^{k+1}_b(\Rn)$ and each $h \in L^q_k(\Rn)$ we define $\delta_x(f):=f(x)$ and $M_{\xi}(h):= e_{\xi}h$. 
  We define $\tilde{\delta}: \RnRnRn \rightarrow \mathscr{L} (C^{k+1}_b(\Rn), \C)$, $\tilde{G}: \RnRnRn \rightarrow L^q_k(\Rn)$ and $\tilde{M}: \RnRnRn \rightarrow \mathscr{L} ( L^q_k(\Rn), L^q(\Rn) )$ by
  \begin{align*}
    \tilde{\delta}(x,y,\xi):= \delta_{x}, \; \tilde{G} (x,y,\xi):= g_y, \; \tilde{M}(x,y,\xi):= M_{\xi} \qquad \text{for all } x,y,\xi \in \Rn. 
  \end{align*}
  One can show that $\tilde{G}$ is a smooth function and that $\tilde{\delta}, \tilde{M}$ are k-times continuous differentiable, cf.\,\cite[Proposition 5.33 and Proposition 5.34]{Diss}. 
  On account of the product rule 
  we get
  \begin{align}\label{p16}
    \tilde{M}(x,y,\xi) \circ \tilde{G}(x,y,\xi) 
    \in C^k( \Rn_x \times \Rn_y \times \Rn_{\xi}, L^q(\Rn)). 
  \end{align}
  Since 
  $T_{\e} \in \mathscr{L} (L^q(\Rn), C_b^{k+1}(\Rn))$, cf.\,\cite[Lemma 5.29]{Diss},
  we obtain 
  \begin{align*}
    T_{\e} (\tilde{M}(x,y,\xi) \circ \tilde{G}(x,y,\xi)) \in C^k( \Rn_x \times \Rn_y \times \Rn_{\xi}, C_b^{k+1}(\Rn))
  \end{align*}
  due to (\ref{p16}). Applying 
  the product rule
  again yields
  \begin{align*}
      p_{\e,0} (x,y,\xi)=e^{-ix \cdot \xi} \tilde{\delta}(x,y,\xi) \circ T_{\e} (\tilde{M}(x,y,\xi) \circ \tilde{G}(x,y,\xi)) \in C^k( \Rn_x \times \Rn_y \times \Rn_{\xi} ).
  \end{align*}
\end{proof}

%% file: SymbolComposition.tex
\section{Composition of Pseudodifferential Operators}\label{section:AnwendungCaracterization}

A calculus for non-smooth pseudodifferential operators in the non-smooth symbol-class $C^{\tau} S^{m}_{1,\delta}(\R^{n}\times\R^{n})$ was first developed by Kumano-Go and Nagase in \cite{KumanoGoNagase}. Meyer and Marschall improved this calculus in \cite{Meyer} and \cite[Chapter\,\,6]{MarschallThesis}. Later Marschall adapted the arguments given there to obtain a calculus for the general case $C^{\tau}\Snn{m}{\rho}{\delta}{N}$ in \cite{Marschall}. Moreover a calculus for operator-valued pseudodifferential operators with non-smooth symbols of class $C^{\tau}\Snn{m}{1}{0}{\mathscr{L}(X_1,X_2)}$ was treated by A.\,in \cite{AbelsPseudoPaper}.

We recall that the composition of two smooth pseudodifferential operators is also a smooth pseudodifferential operator, cf.\;e.\;g.\;\cite[Theorem 3.16]{PDO}. 
But in contrast to the smooth case, the composition of two non-smooth pseudodifferential operators is in general not a pseudodifferential operator with the same regularity with respect to its coefficient, cf.\;\cite[p.1465]{AbelsPseudoPaper}. To illustrate this, let $p\in C^{\tau} S^{m}_{1,0}(\R^{n}_x \times\R^{n}_{\xi})$ with $\tau \in (0,1)$, $m \in \R$ and $p(x,\xi) \notin C^1(\Rn_x)$ for all $\xi \in \Rn$ and $j \in \{ 1, \ldots, n\}$.  
Then $p(x,D_x)D_{x_j}  =\op(p(x,\xi)\xi_j)$ is an element of $ \op C^{\tau} S^{m+1}_{1,0}(\R^{n}\times\R^{n})$. 
Therefore the question arises whether $\op (\xi_j)p(x,D_x)$ is also a pseudodifferential operator. If this would be the case, the iterated commutator 
\begin{align}
 \ad(D_{x_j})p(x,D_x)  = \op (\xi_j)p(x,D_x) - p(x,D_x) \op(\xi_j)= (\p_{x_j} p)(x,D_x)
\end{align}
would be a pseudodifferential operator, too. Hence $\p_{x_j} p \in C^0(\Rn)$ for all $j \in \{ 1, \ldots, n\}$ and $p(.,\xi) \in C^1(\Rn)$ for all $\xi \in \Rn$, which contradicts the assumptions.
Consequently $\ad(D_{x_j})p(x,D_x)$ is not a pseudodifferential operator
just like $\op(\xi_j)p(x,D_x)$.

With the characterization of non-smooth pseudodifferential operators at hand, we can prove a result for the composition of two non-smooth pseudodifferential operators:

\begin{thm}\label{thm:symbolComposition}
  Let $m_i \in \R$, $M_i \in \N \cup \{ \infty \}$ and $\rho_i \in \{0,1\}$ for $i \in \{ 1,2 \}$. Additionally let $0< \tau_i < 1$ and $\tilde{m}_i \in \N_0$ be such that $\tau_i + \tilde{m}_i > (1-\rho_i)n/2$ for $i \in \{ 1,2 \}$. We define $k_{i}:=(1-\rho_i)n/2$ for $i \in \{ 1,2\}$, $\rho:= \min\{\rho_1; \rho_2 \}$ and $m:= m_1+ m_2 + k_{1} + k_{2}$.   
  Moreover, let $\tilde{m}, M \in \N$ and $1< q < \infty$ be such that 
  \begin{itemize}
    \item[i)] $M \leq \min \left\{ M_i- \max \{ n/q; n/2 \}: i \in \{1,2\} \right\}$,
    \item[ii)] $n/q < \tilde{m} \leq \min\{ \tilde{m}_1 ; \tilde{m}_2\}$,
    \item[iii)] $\tilde{m} < \tilde{m}_2 + \tau_2  - m_1 - k_{1}$,
    \item[iv)] $\rho M + \tilde{m} < \tilde{m}_2 + \tau_2 + m_1  + k_{1}$,
    \item[v)] $\tilde{M} \geq 1$, where $\tilde{M}:= M-(n+1)$, 
    \item[vi)] $q=2$ in the case $(\rho_1, \rho_2) \neq (1,1)$.
  \end{itemize}
  Considering two symbols $p_i \in C^{\tilde{m}_i, \tau_i} \Snn{m_i}{\rho_i}{0}{M_i}$, $i \in \{1,2\}$, we obtain
  \begin{align*}
    p_1(x,D_x) p_2(x,D_x) \in \op C^{\tilde{m}-n/q} \Snn{m}{\rho}{0}{\tilde{M}-1}.
  \end{align*}
\end{thm}

\begin{proof}
  Let $l \in \N$, $\varf{\tilde{\alpha}}{l} \in \Non$ and $\varf{\tilde{\beta}}{l} \in \Non$ with $|\tilde{\alpha}| \leq M$, $|\tilde{\beta}| \leq \tilde{m}$ and $|\tilde{\alpha}_1| + |\tilde{\beta}_1| = \ldots = |\tilde{\alpha}_l| + |\tilde{\beta}_l| = 1$ be arbitrary. Here $\tilde{\alpha}:= \vara{\tilde{\alpha}}{l}$ and $\tilde{\beta}:= \vara{\tilde{\beta}}{l}$. Due to Remark \ref{bem:SymbolOfIteratedCommutatorNonSmooth}, $i)$ and $ii)$ we know that
  \begin{align*}
    \ad(-ix)^{\tilde{\alpha}_l} \ad(D_x)^{\tilde{\beta}_l} \ldots \ad(-ix)^{\tilde{\alpha}_1} \ad(D_x)^{\tilde{\beta}_1} p_i(x,D_x)
  \end{align*}
  is a pseudodifferential operator with symbol 
  $$\pa{\tilde{\alpha}} D_x^{\tilde{\beta}} p_i \in C^{\tilde{m}_i-|\tilde{\beta}|, \tau_i} \Snn{m_i-\rho_i |\tilde{\alpha}|}{\rho_i}{0}{M_i-|\tilde{\alpha}|}$$
  for $i \in \{1,2\}$. Since $ 0 < \tilde{m}_1 - |\tilde{\beta}| + \tau_1 $ because of $ii)$, an application of Theorem \ref{thm:stetigInHoelderRaum} if $\rho_1=1$ and Theorem \ref{thm:stetigInHoelderRaum00} else provides for all elements $u$ of $ H^{m_1-\rho|\tilde{\alpha}|+ k_{1}}_q(\Rn)$:
  \begin{align}\label{eq30}
    \| \ad(-ix)^{\tilde{\alpha}_l} \ad(D_x)^{\tilde{\beta}_l} \ldots &\ad(-ix)^{\tilde{\alpha}_1} \ad(D_x)^{\tilde{\beta}_1} p_1(x,D_x) u\|_{L^q} \leq C \|u\|_{H^{m_1-\rho_1|\tilde{\alpha}|+ k_{1}}_q} \notag \\
    & \qquad \qquad \leq C \|u\|_{H^{m_1-\rho|\tilde{\alpha}|+ k_{1}}_q}.
  \end{align}
  Now let $k \in \N_0$ with $k \leq M$ be arbitrary. Making use of estimate $iii)$ and $iv)$
  we are able to verify that the assumptions of 
  Theorem \ref{thm:stetigInHoelderRaum} hold. An application of Theorem \ref{thm:stetigInHoelderRaum} in the case $\rho_1=1$ and Theorem \ref{thm:stetigInHoelderRaum00} else yields for all $u \in H^{m-\rho (k+|\tilde{\alpha}|)}_q(\Rn)$:
  \begin{align}\label{eq31}
    \| \ad(-ix)^{\tilde{\alpha}_l} \ad(D_x)^{\tilde{\beta}_l} \ldots &\ad(-ix)^{\tilde{\alpha}_1} \ad(D_x)^{\tilde{\beta}_1} p_2(x,D_x) u\|_{H^{m_1-\rho k+ k_{1}}_q} \notag \\
    &\leq C \|u\|_{H^{m-\rho k - \rho_2 |\tilde{\alpha}|}_q} 
    \leq C \|u\|_{H^{m-\rho (|\tilde{\alpha}| + k)}_q}.
  \end{align}
  We assume arbitrary $l \in \N$, $\varf{\alpha}{l} \in \Non$ and $\varf{\beta}{l} \in \Non$ with $|\alpha| \leq M$, $|\beta| \leq \tilde{m}$ and $|\alpha_1| + |\beta_1| = \ldots = |\alpha_l| + |\beta_l| = 1$. Here $\alpha:= \vara{\alpha}{l}$ and $\beta:= \vara{\beta}{l}$. Using (\ref{eq30}) and (\ref{eq31}) we obtain iteratively
  \begin{align*}
    &\left\| \ad(-ix)^{\alpha_l} \ad(D_x)^{\beta_l} \ldots \ad(-ix)^{\alpha_1} \ad(D_x)^{\beta_1} [ p_1(x,D_x) p_2(x,D_x) ] u \right\|_{L^q} \\
    &\qquad \leq C \sum_{ \substack{ \tilde{\alpha}_j + \gamma_j = \alpha_j \\ \tilde{\beta}_j + \delta_j = \beta_j} }
      \| [\ad(-ix)^{\tilde{\alpha}_l} \ad(D_x)^{\tilde{\beta}_l} \ldots \ad(-ix)^{\tilde{\alpha}_1} \ad(D_x)^{\tilde{\beta}_1} p_1(x,D_x)] \\
      &\qquad \qquad \qquad \quad \qquad \left. [\ad(-ix)^{\gamma_l} \ad(D_x)^{\delta_l} \ldots \ad(-ix)^{\gamma_1} \ad(D_x)^{\delta_1} p_2(x,D_x)]u \right\|_{L^q} \\
    &\qquad \leq C \|u\|_{H^{m-\rho |\alpha|}_q} \qquad \text{for all } u \in H^{m-\rho |\alpha|}_q(\Rn).
  \end{align*}
  Consequently $p_1(x,D_x) p_2(x,D_x) \in \mathcal{A}^{m,M}_{\rho,0}(\tilde{m},q)$. Due to $ii)$ and $v)$ we can apply Theorem \ref{thm:classificationA10}in the case $\rho=1$ and Lemma \ref{lemma:classA00} if $\rho=0$ and we get 
  the claim.
\end{proof}